\documentclass[psamsfonts]{amsart}
\usepackage{
graphicx,
fancyhdr,
amsmath,
amsfonts,
dsfont,
amsthm,
amssymb,
enumerate,
mathrsfs,
hyperref,
xcolor,
ulem,
cancel,
enumitem,
cleveref
}

\newcommand{\bpf}[1][Proof]{{\noindent {\sc #1: }}}
\newcommand{\epf}{{{\hfill $\Box$ \smallskip}}}

\newcommand{\stepdensity}{p}
\newcommand{\E}{\mathbb{E}}
\newcommand{\R}{\mathbb{R}}

\newcommand{\Z}{\mathbb{Z}}

\newcommand{\N}{\mathbb{N}}
\newcommand{\bb}[1]{\mathbf{#1}}
\newcommand{\1}{\mathds{1}}

\newcommand{\Prb}{\mathbb{P}}
\newcommand{\QQ}{\mathbb{Q}}

\newcommand{\Bc}{\mathcal{B}}
\newcommand{\momentnumber}{\nu}
\newcommand{\logExp}{\vartheta}
\newcommand{\freeEnergy}{\rho}
\newcommand{\X}{\mathbb{X}}
\newcommand{\Ofinite}{\Omega_1} 
\newcommand{\OfinitePoint}{\Omega_2}
\newcommand {\OfiniteBoth}{\Omega_0}
\newcommand{\densityRatio}{r_0}

\newtheorem{proposition}{Proposition}[section]
\newtheorem{corollary}{Corollary}[section]
\newtheorem{lemma}{Lemma}[section]
\newtheorem{theorem}{Theorem}[section]

\theoremstyle{definition}
\newtheorem{definition}{Definition}[section]
\newtheorem{assumption}{Assumption}
\renewcommand*{\theassumption}{\Alph{assumption}}
\newtheorem{remark}{Remark}

\newlist{assumpRequirements}{enumerate}{10}
\setlist[assumpRequirements]{label*=\Roman*}
\setlist[assumpRequirements,1]{label=(\Roman*), ref = \theassumption.\Roman*}
\crefname{assumpRequirementsi}{assumption}{assumptions}

\newlist{requirements}{enumerate}{10}
\setlist[requirements]{label*=\alph*}
\setlist[requirements,1]{label=\rm(\alph*), ref = \rm(\alph*)}
\crefname{requirementsi}{requirement}{requirements}

\numberwithin{equation}{section}

\author{Yuri Bakhtin and Douglas Dow}
\title{Joint Localization of Directed Polymers}

\begin{document}

\begin{abstract}
    We consider $(1+1)$-dimensional directed polymers in a random potential and provide sufficient conditions guaranteeing joint localization. Joint localization means that for typical realizations
of the environment, and for polymers started at different starting points, all the associated endpoint distributions localize
in a common random region that does not grow with the length of the polymer. In particular, we prove that joint localization holds when the reference random walk of the polymer model is either a simple symmetric lattice walk or a Gaussian random walk. We also prove that
the very strong disorder property holds for a large class of space-continuous polymer models, implying the usual single polymer localization. 
\end{abstract}

\maketitle

\section{Introduction}

The term {\it directed polymers} refers to a class of models describing a random elastic chain in $\R^d$ interacting with its random environment. The random distributions on these chains or paths are given by random Gibbs measures, with Hamiltonian composed of the energy of local self-interaction and the energy of interaction with the environment.  

Pinning one of the endpoints of the polymer chain and parametrizing the chain by time, one can view it as a {\it random walk in random potential.} 
Time is a distinguished coordinate in this $(1+d)$-dimensional model. 
In the absence of interaction with the environment, i.e., when the external potential is zero, random walks are diffusive: the 
distribution of the 
free endpoint of the path of length $n$ is approximately Gaussian with variance of order $n$, so it spreads out and thins out as $n\to\infty$.

In the presence of an external potential,  favorable and unfavorable regions for polymers are created randomly, and one of the central questions in the theory is how these impurities change the diffusive behavior, under the assumption that the environment evolves and decorrelates in time.

This assumption implies that different parts of the polymer path are exposed to different states of the environment.  Attractive regions and other environment features are dynamically created and destroyed at every time step, and so it is fascinating that, despite this, the random distribution of the polymer endpoint is often localized. In other words, for large $n$ it does not spread out like a Gaussian distribution with large variance. 

Conditions guaranteeing various forms of localization have been extensively studied in the literature over the last two decades,
 see \cite{CH02}, \cite{CSY03},
\cite{Comets-Shiga-Yoshida:MR2073332}, 
\cite{Carmona-Hu:MR2249669}, \cite{Comets-Shiga-Yoshida:MR2073332}, \cite{Comets-Yoshida:MR2117626}, \cite{Rovira-Tindel:MR2129770}, \cite{CV06}, \cite{Var07}, \cite{Lac10},   \cite{Comets-Cranston:MR3038513}, \cite{Comets-Yoshida:MR3085670},
\cite{Bates-Chatterje:MR4089496}, \cite{Bat18}, \cite{Broker-Mukherjee:MR4047991}, \cite{bakhtin-seo2020localization}, \cite{Bates:MR4269204},
\cite{Das-Zhu:https://doi.org/10.48550/arxiv.2203.03607}. Localization for directed polymers is one of the main themes of the monograph
 \cite{Comets:MR3444835}.

In these works, various manifestations of localization were studied such as presence of uniformly heavy atoms for the endpoint distribution, its asymptotic pure atomicity, asymptotically nonvanishing replica overlap, and similar stronger notions in terms of entire paths. The new notion of {\it geometric localization} introduced for lattice models in  \cite{Bates-Chatterje:MR4089496}, \cite{Bat18}, and studied in \cite{bakhtin-seo2020localization} for continuous space models  (along with the several new notions of asymptotic clustering replacing asymptotic pure atomicity
studied for lattice systems in \cite{Var07},  \cite{Bates-Chatterje:MR4089496},  \cite{Bat18})   
essentially means that despite the growing length of the polymer,
its endpoint distribution mostly concentrates in a random region of size of order 1. This is, of course, in sharp contrast with the diffusive behavior observed for classical symmetric random walks. Localization is related to the phenomenon of intermittency for solutions of the stochastic heat equation, see
\cite{CM94}, \cite{BC95}, \cite{Kho14}.  The main factors contributing to the localization/delocalization are the strength of the random potential in relation to the temperature and the dimension $d$.

With every polymer model one can associate its Lyapunov exponent  characterizing the discrepancy between the density of quenched and annealed free energies, see precise definitions below.  One says that {\it very strong disorder} holds if the Lyapunov exponent is strictly positive. Results of \cite{Bates-Chatterje:MR4089496} and 
their generalizations in~\cite{bakhtin-seo2020localization} establish that very strong disorder is equivalent to 
geometric localization. In particular, applying results of \cite{CV06} and \cite{Lac10}, we obtain that geometric localization holds for a broad class of lattice models with i.i.d.\ environments in dimensions $1+1$ and $1+2$ and all temperature values. It is known that very strong disorder holds only for sufficiently low temperatures in dimensions $d\ge 3$: in this case, the so-called {\it weak disorder} holds for high temperatures,
see \cite{IS88}, \cite{Bol89}, \cite{Albeverio-Zhou:MR1371075}, \cite{Sinai:MR1341729}, \cite{Kifer:MR1452549}. It is conjectured in \cite{BK18} that a similar picture holds for a broader class of generalized Hamilton--Jacobi polymers.

A remarkable explicit representation for the limiting random probability density 
of the endpoint (and intermediate points) of the space-time white noise continuous polymers in $1+1$ dimension was recently obtained in \cite{Das-Zhu:https://doi.org/10.48550/arxiv.2203.03607}.

\bigskip

We are interested in the joint behavior of polymer measures exposed to  the same environment but with different endpoints. Such  polymers, 
in dimension $1+1$,
with Gaussian random walk as a reference measure, i.e., with quadratic nearest neighbor self-interaction, play an important 
role in the ergodic theory of the heat/Burgers/KPZ equation with random kick forcing: one can construct attracting global random solutions in terms of Busemann functions associated with thermodynamic limits of polymer measures with different endpoints, see~\cite{Bakhtin-Li:MR3911894}
where all these objects were constructed. We will refer to them as GRW (Gaussian random walk) polymers in this paper.

It is natural to conjecture that if localization holds for individual polymers, then polymers with different endpoints exposed to the same environment must be localized in the same region, i.e.,
one can find a random set of size of order~$1$ containing most of the mass of both endpoint distributions, of both polymers pinned at different endpoints.

Our main result is that this is true in the $(1+1)$-dimensional case under natural additional assumptions. We also check these additional assumptions and derive that joint localization indeed holds for two classes of polymers in $1+1$ dimensions. The first class of models is the GRW polymers studied in~\cite{Bakhtin-Li:MR3911894}. The second class is the lattice polymers with symmetric simple random walk (SSRW) as the reference measure, in an i.i.d.\ environment.  
 
We note that while the very strong disorder for lattice polymers was established in~\cite{CV06} and similar results have been 
obtained for some continuous-space polymers
in \cite{Comets-Yoshida:MR2117626}, \cite{Rovira-Tindel:MR2129770}, \cite{Broker-Mukherjee:MR4047991}, \cite{Das-Zhu:https://doi.org/10.48550/arxiv.2203.03607}, the same property for 
GRW polymers has not appeared in the literature to the best of our knowledge.
In this note, we adapt the proof from~\cite{CV06} and derive that very strong disorder holds for a large class of space-continuous polymers, including GRW polymers, which in conjunction with results from \cite{Bakhtin-Li:MR3911894}, \cite{bakhtin-seo2020localization}, and our main result allows us to conclude that joint localization holds for GRW polymers for all temperature values in dimension $1+1$.

Besides the localization for individual polymers, an additional assumption we need is closely related to the requirement that the ratio of certain point-to-line partition functions remains bounded by a random constant not depending on the polymer length.
A stronger form of this property was established for
GRW polymers in~\cite{Bakhtin-Li:MR3911894}, where convergence of these ratios to finite positive numbers was established. The limits of those ratios or their logarithms can be viewed as Busemann functions, and their existence has been also established for exactly solvable models of log-gamma polymer in \cite{Georgiou--Rassoul-Agha--Seppalainen--Yilmaz:MR3395462} and O'Connell--Yor polymer 
in~\cite{Alberts--Rassoul-Agha--Simper:MR4091097}. For general SSRW models, convergence of partition function ratios (i.e., well-definedness of Busemann functions)
is known only conditionally, see~\cite{Janjigian--Rassoul-Agha:MR4089495}, but the boundedness of these ratios can be derived from the latter paper.

A first indication that such a result could be true is a theorem from~\cite{Bakhtin-Li:MR3911894} stating convergence to zero (as $n\to\infty$) of the total variation distance between the $n$-th marginals of the infinite-volume polymer measures
with the same asymptotic slope but with different endpoints.

\textbf{Acknowledgements.} We are grateful to Firas Rassoul-Agha for pointing out that our condition on the boundedness of partition functions can be easily derived from \cite{Janjigian--Rassoul-Agha:MR4089495}. YB thanks NSF for partial support via grant DMS-1811444.

\section{The setting and statements of main results}\label{Setting_MainResults}

 We will consider two main cases, lattice polymers and polymers in continuous space. To unify the notation, we define the space $\X$ to be either $\Z$ or $\R$.  In both cases, $\X$ is a group with respect to the usual addition. The role of space-time for our polymers in $1+1$ dimension is played by $\Z\times\X$ ($\X$ being the space and $\Z$ being the time).

Let $\lambda$ be a probability measure on $(\X,\mathcal{B}),$ where $\mathcal{B}$ is the Borel $\sigma$-algebra on~$\X$. To simplify the presentation we assume that $\lambda$ is not equal to a Dirac mass $\delta_x$ for any $x\in \X$. For $a\in\X$ and an i.i.d.\ sequence of random variables $(\xi_i)_{i\in \N}$ with common distribution $\lambda$ on $\X$, we define the partial sums $S_0=a$,
$S_k = a+\sum_{i = 1}^{k}\xi_i$ for $k\ge 1$, and denote the distribution of $(S_0,\dots, S_{n-1},S_n)$ on $\X^{n+1}$ by $\mathrm{P}_a^n$. Equivalently, 

\[ \mathrm{P}_a^n(dx_0,\dots, d x_{n}) = \delta_a(dx_0) \lambda(d(x_1-x_0))\cdots \lambda(d(x_{n} - x_{n-1})). \]
For integers $m < n$ we use the notation $\mathrm{P}_a^{m,n}:= \mathrm{P}_a^{n-m}$. In this paper, we consider probability measures $\lambda$ of the form $\lambda(dx)=\stepdensity(x)\gamma(dx) $, where $\gamma$ is either Lebesgue measure on $\X=\R$ or the counting measure on $\X=\Z$ and $\stepdensity$ is a probability density, i.e., a measurable nonnegative function such that $\int_\X \stepdensity(x)\gamma(dx)=1$.

We state our main results  (see \Cref{thm:generalThm}) imposing extra conditions on the distribution $\lambda$.
Although we believe that these conditions hold true for a large class of models, in this paper, we verify them only for two specific examples: the Gaussian random walk,
where  $\lambda(dx) = g(x)dx$ with 
\begin{equation}
\label{eq:standard_gaussian_d}
g(x) = \frac{1}{\sqrt{2\pi}}e^{-x^2/2},
\end{equation}
 and the symmetric simple random walk, where  $\lambda(dx) = \frac{1}{2}\delta_{-1} + \frac{1}{2}\delta_1$, see our Assumptions~\ref{continuousEnvironmentAssumptions} and~\ref{assumptions2} below.

The $(1+1)$-dimensional polymer measures that we are concerned with are obtained as a result of interaction of the above random walks  with the environment potential, a measurable function $F:\Z\times\X\to \R$ which can be viewed as a collection $F=(F_k)_{k\in\Z}$, where $F_k:\X\to\R$, $k\in\Z$.

We denote the point-to-line directed polymer measure at temperature $T=1/\beta\in(0,\infty)$, in potential $F$, started at $(m,a)\in \Z\times \X$ and ending at time $n>m$, by $\mu_{a,\beta}^{m,n}$. 
 The measure~$\mu_{a,\beta}^{m,n}$ is a Gibbs distribution. It is defined to be the probability measure on $\X^{n-m+1}$ satisfying
\begin{equation}
\label{eq:p2l-measure}
\mu_{a,\beta}^{m,n}(d x_m,\dots, d x_{n}) 
= \frac{1}{Z_{a,\beta}^{m,n}} e^{-\beta \sum_{k = m}^{n-1} F_k(x_k)} \mathrm{P}_{a}^{m,n} (dx_m,\dots, d x_{n}),
\end{equation}
where the normalizing constant $Z_{a,\beta}^{m,n}$ called the point-to-line partition function is defined by
\begin{equation}
\label{eq:p2l-Z}
Z_{a,\beta}^{m,n} = \int_{\X^{n-m+1}}e^{-\beta \sum_{k = m}^{n-1} F_k(x_k)} \mathrm{P}_{a}^{m,n} (dx_m,\dots, dx_{n}).
\end{equation}
We note that \eqref{eq:p2l-measure}  defines a probability measure $\mu_{a,\beta}^{m,n}$ whenever $Z_{a,\beta}^{m,n}<\infty$.
 
Since in this paper we only consider $1+1$ dimension, the role of $\beta$ is inconsequential, as the very strong disorder regime holds for all positive $\beta$ and environments satisfying certain mild conditions (see \cite{CV06} for the discrete case and Section~\ref{veryStrongDisorderSection} for the continuous case). Because of this, we can omit the constant $\beta$ from our notation and absorb it into the environment $F$. We also use $Z^n_a=Z^{0,n}_a$,  $\mu^n_a=\mu^{0,n}_a$, $Z^n=Z^n_0$, $\mu^n=\mu^n_0$ for brevity.

The polymer measure $\mu_{a}^{m,n}$ can be viewed a Gibbs measure with reference measure  $\delta_a\times \gamma^{n-m}$ and Hamiltonian 
\[
\mathcal{H}^{m,n}_{x_m}(x_{m+1},\dots,x_{n-1}) = \sum_{k = m}^{n-1} \Big[F_k(x_k) + V(x_{k+1} - x_k)\Big],
\]
where
\begin{equation}
    V(x) := -\log \stepdensity(x).\label{potentialDef}
\end{equation}
Under this interpretation, $V(x_{k+1}-x_k)$ plays the role of the energy of nearest neighbor self-interaction of the polymer chain. For the Gaussian density of the random walk step 
given in~\eqref{eq:standard_gaussian_d}, the energy $V$ is quadratic.

We assume we are given a probability space $(\Omega, \mathcal{F},\Prb)$, where $\Omega$ is the space of continuous functions $F:\Z\times \X\to \R$ 
(the continuity requirement is superfluous if $\X=\Z$) 
endowed with local uniform topology and $\mathcal{F}$ is the completion of the Borel $\sigma$-algebra with respect to $\Prb$.  

We usually write the time argument of $F$  as a subscript obtaining $F_k:\X\to \R$, $k\in\Z$. We introduce the space-time shifts $(\theta^{n,x})_{n\in \Z, x\in \X}$ acting on $\Omega$, defined by $\theta^{n,x}F_k(y) = F_{k+n}(y+x)$, and we assume that these space-time shifts preserve $\Prb$ so that $F$ is space-time stationary. We take the collection $(F_k)_{k\in \Z}$ to be independent. 
We also assume throughout this paper that $F$ is not deterministic (a deterministic stationary potential $F$ is a constant, so in this case the polymer measures coincide with
the reference random walks and localization does not hold).

We are studying directed polymer measures $\mu_{a}^{m,n}$ in environment $F$. Though~$\mu_{a}^{m,n}$ depends on the realization of the environment, this dependence will be omitted from the notation for brevity. 

We are interested the joint behavior of polymer measures $\mu_{a}^{m,n}$ with varying  $a,m,$ and $n$.
To ensure that they all are well-defined probability measures for almost every realization of $F$, we will need to check that 
\begin{equation}
    \Prb(\Ofinite) = 1\label{almostSureOFinite}
\end{equation}
where
\begin{equation}
    \Ofinite := \bigcap_{\substack{m,n\in \Z\\m<n}} \{Z_a^{m,n}< \infty\ \text{for all\ } a\in\X \}.\label{defOfOfinite}
\end{equation}
We discuss this condition in Section~\ref{almostSureFinitenessSection}. In particular, \Cref{almostSureFinitenessCorollary} implies that~\eqref{almostSureOFinite} holds in the concrete cases that we consider.

Some of our results require additional sets of conditions that we collect together as Assumption~\ref{continuousEnvironmentAssumptions} (for the continuous space case) and Assumption~\ref{assumptions2} (for the lattice case).
In part these assumptions are carried over from \cite{Bakhtin-Li:MR3911894,bakhtin-seo2020localization,Janjigian--Rassoul-Agha:MR4089495,Bates-Chatterje:MR4089496} to ensure that the main results from these papers are applicable. More precisely,
\Cref{continuousEnvironmentAssumptions} fulfills requirements on the environment stated in \cite{Bakhtin-Li:MR3911894} and \cite{bakhtin-seo2020localization} on the continuous space case. In this paper, we need an additional assumption of positive correlation, as well as a slightly more restrictive exponential moment condition than in \cite{Bakhtin-Li:MR3911894} and \cite{bakhtin-seo2020localization}. \Cref{assumptions2} is a combination of requirements in \cite{Janjigian--Rassoul-Agha:MR4089495} and \cite{Bates-Chatterje:MR4089496} in the lattice case. 

\begin{assumption}[Continuous Environment Case]\label{continuousEnvironmentAssumptions}
    Here $\X=\R$, $\gamma(dx)$ is Lebesgue measure, and $\lambda$ is absolutely continuous with respect to Lebesgue measure. The requirements on the environment are the following.
    \begin{assumpRequirements}
        \item For all $\alpha\in[-2,3]$, $\E[e^{\alpha F_0(0)}] < \infty$.  In addition, there exists $\eta > 0$ such that $\E[e^{\eta F^*_0(0)}] < \infty$, where $F^*_k(x) = \sup\{F_k(y)\,:\,y\in [x,x+1]\}$.\label{exponentialMomentAssumption}
        \item Positive correlation, i.e., for all $x\in \R,$
        \[ \E\left[ \left(e^{ -F_0(0)} - e^{\logExp}\right)\left(e^{ -F_0(x)} - e^{\logExp}\right)\right] \ge 0\]
        where $\logExp := \log \E[e^{-F_0(0)}].$
        \item $\Prb$-almost surely, $F_k\in C^1(\R)$ for all $k$.
        \item $F_0$ is $M$-dependent for some $M>0$, i.e., for all  $a\in\R$, the collection of random variables $(F_0(x))_{x< a}$ is independent of $(F_0(x))_{x>a+M}$.
    \end{assumpRequirements}
\end{assumption}

Natural examples of a field $F$ satisfying the above conditions are: (i) a Poisson field on $\Z\times\R$ mollified in the spatial variable by a compactly supported nonnegative $C^1$ function;
(ii) a smooth Gaussian field with positive covariance, with finite dependence range in space and i.i.d.\ in time.

\begin{assumption}[SSRW Case]\label{assumptions2}
Here $\X=\Z$ and $\lambda(dx) = \frac{1}{2}\delta_{-1} + \frac{1}{2}\delta_1$. The requirements on the environment are the following.
\begin{assumpRequirements}
\item For all
 $\alpha \in [-2,2]$, $\E\left[ e^{\alpha F_0(0)} \right]< \infty.$
\item The collection $(F_0(x))_{x\in \Z}$ is i.i.d.
\end{assumpRequirements}
\end{assumption}

Now we will define precisely the concepts of localization and joint localization that we will consider. The following definition from~\cite{bakhtin-seo2020localization} is a generalization of the definition for lattice measures given in~\cite{Bates-Chatterje:MR4089496}. For brevity,  we replace the term  {\it geometric localization} used in those papers by {\it localization}. For $x\in\X$, $K>0$, we denote $B_K(x)=\{y:|y-x|\le K\}$, the closed ball of radius~$K$ centered at $x\in \X.$

\begin{definition}\label{def:localiz}
Localization with parameters $(\delta, K, \theta)$ holds for a sequence $(\nu^n)_{n\in \N}$ of probability measures on $\R$  if
\[\liminf_{n\to \infty}\frac{1}{n}\sum_{k = 0}^{n-1}\1\Big\{\sup_{x\in \R}\nu^k (B_K(x)) > 1 - \delta\Big\} \ge \theta.\]
\end{definition}
In other words, for every $k$ from a set of natural numbers of density at least~$\theta$, the measure $\nu^k$ assigns mass at least $1-\delta$  to some region of fixed size $2K$. In this paper, we study a related notion of joint localization defined as follows.
\begin{definition}\label{def:joint-localiz} 
Let $A\subset \X$ be any set.
Joint
localization with parameters $(\delta,K,\theta)$ holds for a family of probability measures $(\nu_a^n)_{n\in \N, a\in A}$  if
\[\liminf_{n\to \infty}\frac{1}{n}\sum_{k = 0}^{n-1}\1\Big\{\sup_{x\in \R}\inf_{a\in A}\nu_a^k(B_K(x)) > 1 - \delta\Big\} \ge \theta.\]
\end{definition}
If $A$ contains more than one point, then this definition
strengthens Definition~\ref{def:localiz} on localization of individual measures and requires 
a uniformly heavy region $B_K(x)$ of fixed size $2K$ to exist and to serve all the measures $\nu_a^k$, $a\in A$, at the same time, i.e., all these measures get localized to the same region.

The main results of this paper give sufficient conditions for joint localization to hold for endpoint distributions $\rho^n_a$ of random polymer measures $\mu^{n}_a$   
defined by 
\begin{equation}
\label{eq:endpoint-distr}
\rho^n_a=\mu^{n}_a\pi^{-1}_n,\quad n\in\N.
\end{equation}
Here and throughout the paper,  $\pi_n x$, $n\in\Z$ denotes the  $n$-th coordinate of a vector (or path) $x$.

Our analysis of point-to-line polymer measures relies on analysis of the point-to-point polymer measures. To define the point-to-point polymer measure we must define the point-to-point reference walk distribution. For $a,u\in \X$ and $n\in \N$, the random walk $\mathrm{P}_{a,u}^{n}$ measure between points $(a,0)$ and $(u,n)$ is defined by
\[ \mathrm{P}_{a,u}^{n}(dx_0,\dots, d x_{n}) = \delta_a(dx_0) \lambda(d(x_1-x_0))\cdots \lambda(d(x_{n} - x_{n-1}))\delta_u(dx_n).\]
The measure $\mathrm{P}_{a,u}^n$ has density with respect to $\delta_{a}\otimes \gamma^{\otimes (n-1)}\otimes \delta_u$ given by
\[
\frac{d \mathrm{P}_{a,u}^{n}}{ d (\delta_{a}\otimes \gamma^{\otimes (n-1)}\otimes \delta_u)}(x_0,\dots, x_{n})=
\prod_{i = 0}^{n-1} \stepdensity(x_{k+1} - x_k).
\]
Note that under our definition $\mathrm{P}_{a,u}^{n}$ is not a probability measure (indeed it may be zero if $(a,0)$ and $(u,n)$ are not connected by a random walk with steps $\lambda$). We use the notation $\mathrm{P}_{a,u}^{m,n} := \mathrm{P}_{a,u}^{n-m}.$ 
We can define the point-to-point polymer measure as 
\begin{equation}
\label{eq:p2p-measure}
\mu_{a,u}^{m,n}(d x_m,\dots, d x_{n}) = \frac{1}{Z_{a,u}^{m,n}} e^{- \sum_{k = m}^{n-1} F_k(x_k)} \mathrm{P}_{a,u}^{m,n} (dx_m,\dots, d x_{n})\end{equation}
where $Z_{a,u}^{m,n}$ is the point-to-point partition function (normalizing factor): 
\begin{align}\label{eq:p2pPartitionFunctionDef}
Z_{a,u}^{m,n} &= \int_{\R^{n-m+1}}e^{- \sum_{k = m}^{n-1} F_k(x_k)} \mathrm{P}_{a,u}^{m,n} (dx_m,\dots, d x_{n}) \\
&= \int_{\R^{n-m+1}}\prod_{k=m}^{n-1}e^{- F_k(x_k)} \stepdensity(x_{k+1}- x_k)(\delta_a\otimes \gamma^{\otimes (n-m-1)}\otimes \delta_u)(dx).
\notag
\end{align}
Definition~\eqref{eq:p2p-measure}  makes sense only if $Z_{a,u}^{m,n}$ is positive and  finite.
If $Z_{a,u}^{m,n} = 0$ then we define $\mu_{a,u}^{m,n}~:=~\bb{0}$ where $\bb{0}$ is the zero measure.
Finiteness of the point-to-point partition functions  is discussed  in Section~\ref{almostSureFinitenessSection}. In particular, \Cref{almostSureFinitenessCorollary} implies that
in the cases we consider, 
\begin{equation}
    \Prb(\OfinitePoint) = 1,\label{AlmostSurep2pFiniteness}
\end{equation}
where 
\begin{equation}
    \OfinitePoint = 
     \bigcap_{\substack{m,n\in \Z\\m<n} } \Big\{Z_{a,u}^{m,n} < \infty\text{ for all }a,u\in \X \Big\}.
\end{equation}
We will work on the set
\[
\OfiniteBoth = \Ofinite \cap \OfinitePoint.
\]
We have $\Prb(\OfiniteBoth)=1$ due to \eqref{almostSureOFinite}
and \eqref{AlmostSurep2pFiniteness}.

We will often use the notation $Z_{a,u}^{n}=Z_{a,u}^{0,n}$ and $\mu_{a,u}^{n}=\mu_{a,u}^{0,n}$.
Let us extend~\eqref{eq:endpoint-distr} and define endpoint distributions for polymers started at a general time:
\begin{equation}
\label{eq:endpoint-distr-any-m}
\rho^{m,n}_a=\mu^{m,n}_a\pi^{-1}_n,\quad m,n\in\Z,\ n>m.
\end{equation}
Using  \eqref{eq:p2l-measure}, \eqref{eq:p2l-Z}, \eqref{eq:p2pPartitionFunctionDef}, \eqref{eq:endpoint-distr-any-m}, we obtain
\begin{equation}
    Z_{a}^{m,n}=\int_{\R}Z_{a,u}^{m,n}\gamma(du)\label{p2l_p2pIntegration},
\end{equation}
so
\[
\rho_a^{m,n}(A)=\frac{1}{Z_a^{m,n}}\ \int_A Z_{a,u}^{m,n}\gamma(du),
\]
or, equivalently,
\begin{equation}
\label{eq:end-density}
\rho_a^{m,n}(du) = \frac{Z_{a,u}^{m,n}}{Z_a^{m,n}}\gamma(du).
\end{equation}

\bigskip

We are ready to state our main general result providing sufficient conditions for joint 
localization. We believe that these conditions apply to a broad class of $1+1$ polymer models. 
We are able to check that they actually hold true in two situations: under \Cref{continuousEnvironmentAssumptions} when $\stepdensity=g$ and under \Cref{assumptions2}. We state the resulting corollaries on localization
for GRW polymers and for SSRW polymers as Theorems~\ref{thm:mainThm1} and~\ref{thm:mainThm2}. 

Let $G\subset \X$ be a subset unbounded in both positive and negative directions. In practical applications, $G$ will be taken to be an additive subgroup of $\X$ associated to the random walk measure ($\R$ in the continuous case and $2\Z$ in the simple random walk case). In essence, $G$ is the subset of $\X$ where random walks started at two points $a,b\in G$ interact with the same environment. 
\begin{theorem}\label{thm:generalThm} 
Let $A$ be a bounded subset of $G$. Suppose that for some realization of the potential~$F\in \OfiniteBoth$, the following holds.
\begin{requirements}
\item\label{onePointLocalizationReq} For every $\delta > 0$ and every $m\in \Z,$ there are $K,\theta > 0$ such that 
localization with parameters $(\delta, K,\theta)$ holds for 
the sequence $(\rho_{0}^{m,n})_{n=m}^\infty$. 
\item\label{partitionRatioReq} For every $a,b\in G$ satisfying $a\le b$, every $n\in \N$, and $\gamma^{\otimes 2}$-a.e.\ $(u,v)\in\X^2$ satisfying $u\le v$,
\begin{equation}
\label{eq:Z-log-concavity}
Z_{a,u}^n Z_{b,v}^n \ge Z_{a,v}^n Z_{b,u}^n.
\end{equation}
\item \label{positivityMarginalReq} For every $r>0$, there is an integer $m < 0$ such that
\begin{equation}
\label{eq:mass_below_minus}
\liminf_{n\to \infty} (\mu_{0}^{m,n} \pi_0^{-1})((-\infty,-r)\cap G) > 0
\end{equation}
and
\begin{equation}
\label{eq:mass_above_plus}
\liminf_{n\to \infty} (\mu_{0}^{m,n} \pi_0^{-1})((r,\infty)\cap G) > 0.
\end{equation}
\end{requirements}
Then for every $\delta > 0,$ there are $K,\theta > 0$ 
such that 
joint 
localization with parameters $(\delta,K,\theta)$  holds for 
$(\rho_{a}^n)_{n\in \N, a\in A}$.
\end{theorem}

\Cref{onePointLocalizationReq} is the obvious necessary condition that single polymer localization should hold if we want joint localization to hold. 

\Cref{partitionRatioReq} can be interpreted as a path crossing inequality for polymers. The function $-\log Z_{a,u}^n$ is the free energy associated to traveling from $a$ to $u$ in $n$ units of time. 
Relation \eqref{eq:Z-log-concavity} means that the sum of the energies associated to paths between $a$ and $u$ and between $b$ and $v$ is less than the sum of the energies associated to paths between $a$ and $v$ and between $b$ and $u$. For the zero temperature last passage percolation
case, the analogous property for geodesics becomes the path crossing lemma (see Lemma B.2 in \cite{BB021}) and it is related to the cutting corner lemma for Lagrangian minimizers (see Fact 2 and Lemma 3.2 in \cite{ekms:MR1779561}).

We are able to verify \Cref{partitionRatioReq} when the reference measure is log-concave and when the reference measure is a simple random walk.

\Cref{positivityMarginalReq} is an assumption of the non-degeneracy of the marginals of the polymer measure as $n\to \infty.$
It is tightly related to boundedness of partition function ratios $Z_x^{0,n}/Z_{y}^{0,n}$ from infinity and zero as $n\to\infty$. For GRW polymers a stronger condition actually holds: the limits of these ratios are well-defined, positive, and can be interpreted in terms of infinite volume polymer measures, see the discussion in Section~\ref{sec:proof_under_A} and, in particular, Theorem~\ref{thm:partitionRatiosGaussian}. 

We are able to check Requirements \ref{onePointLocalizationReq} and \ref{partitionRatioReq} of \Cref{thm:generalThm} assuming that the random walk density is log-concave, and so we derive the following corollary of \Cref{thm:generalThm}. We call the density $\stepdensity:\R\to [0,\infty)$ log-concave if the energy function $V:\R\to (-\infty,\infty]$
given by~\eqref{potentialDef} is convex. Some well-known examples of log-concave densities are Gaussian densities ($\stepdensity(x)\propto e^{-b (x-a)^2}$), Laplace densities ($\stepdensity(x) \propto e^{-b|x-a|}$), and uniform densities ($\stepdensity(x) \propto \1_{[a,b]}(x)$).  
\begin{corollary}\label{thm:logConcaveCorollary}
    Suppose $\lambda$ has a log-concave density with respect to Lebesgue measure and that the environment satisfies \Cref{continuousEnvironmentAssumptions}. Also, suppose \Cref{positivityMarginalReq} of \Cref{thm:generalThm}
    holds $\Prb$-almost surely for $G=\R.$ Let $\delta > 0$ and $a<b.$ Then, for $\Prb$-almost every realization of~$F$,
    there are $K,\theta>0$ (depending on $F$) such that 
    joint  
    localization with parameters $(\delta, K, \theta)$ holds for the family $(\rho_x^n)_{n\in \N,x\in [a,b]}$.
\end{corollary}

We establish \Cref{positivityMarginalReq}  of \Cref{thm:generalThm} for GRW polymers, implying the following unconditional result.

\begin{theorem}\label{thm:mainThm1}
    Suppose \Cref{continuousEnvironmentAssumptions} holds and $\stepdensity=g$, the standard Gaussian density given in \eqref{eq:standard_gaussian_d}. Let $\delta > 0$ and let $a<b$.  Then, for $\Prb$-almost every realization of~$F$,
    there are $K,\theta>0$ (depending on $F$) such that 
    joint  
    localization with parameters $(\delta, K, \theta)$ holds for the family $(\rho_x^n)_{n\in \N,x\in [a,b]}$.
\end{theorem}

For SSRW polymers,
we are able to verify that \Cref{onePointLocalizationReq,partitionRatioReq,positivityMarginalReq} of \Cref{thm:generalThm} hold $\Prb$-almost surely thus obtaining
the following result.
\begin{theorem}\label{thm:mainThm2}
    Suppose \Cref{assumptions2} holds. Let $\delta > 0$ and let $A$ be a finite subset of $ 2\Z$. Then,
    for $\Prb$-almost every realization of~$F$,
    there are $K,\theta>0$ (depending on~$F$) such  that 
    localization with parameters $(\delta, K, \theta)$ holds for the family $(\rho_x^n)_{n\in \N,x\in A}$.
    \end{theorem}
    
We prove \Cref{thm:generalThm} in Section~\ref{generalPolymersSection}. We prove \Cref{thm:logConcaveCorollary} and \Cref{thm:mainThm1}  in Section~\ref{continuousSpaceSection}. We prove \Cref{thm:mainThm2} in Section~\ref{assumptions2Section}.

The fact that localization holds for single polymers under~\Cref{assumptions2} is a corollary of results establishing very strong disorder~\cite{CV06} and results 
deriving localization from very strong disorder for lattice polymers in \cite{Bates-Chatterje:MR4089496}. However,  no localization results are known to us for
polymers in continuous space and discrete time. We prove such a result in Section~\ref{veryStrongDisorderSection} using ideas from \cite{CV06}
to establish very strong disorder for these models, and applying the main result from \cite{bakhtin-seo2020localization} to derive localization from very strong disorder. 

The very strong disorder property is a strict inequality between the limiting average free energy of the model and the annealed upper bound. Under \Cref{assumptions2}, assuming finiteness of exponential moments of the environment, it was shown in Proposition 2.5 of \cite{CSY03}  that the average free energy,
\begin{equation}
    \freeEnergy := \lim_{n\to \infty}\frac{1}{n}\log Z^n\label{freeEnergyDef}
\end{equation}
exists almost surely and equals $\lim_{n\to \infty} \frac{1}{n}\E[\log Z^n]$. This was improved in \cite{LW09} to include the setting of \Cref{assumptions2}, where fewer exponential moments are assumed. The quantity 
\begin{equation}
    \logExp := \log \E[e^{-F_0(0)}]\label{annealedBoundDef}
\end{equation}
is the annealed bound of the free energy, and upper bounds the free energy by Jensen's inequality. The results from \cite{CV06} and \cite{Lac10} show that under \Cref{assumptions2} the very strong disorder property holds:
\begin{equation}
    \freeEnergy < \logExp.\label{veryStrongDisorder_lattice}
\end{equation}
The quantity $\freeEnergy - \logExp$ is called the Lyapunov exponent of the polymer model. Definitions \ref{freeEnergyDef} and \eqref{annealedBoundDef} can be extended to the continuous setting in the natural way. The existence of the limit in \eqref{freeEnergyDef} in the continuous setting is an easy extension of the analogous result in the discrete setting and is discussed at the beginning of Section~\ref{veryStrongDisorderSection}.

The very strong disorder property \eqref{veryStrongDisorder_lattice}  is equivalent  to localization of the sequence of endpoint measures $\rho_0^n$. Specifically, the following theorem was proved in \cite{Bates-Chatterje:MR4089496} in the context of \Cref{assumptions2} and then extended to the general continuous setting in \cite{bakhtin-seo2020localization}. 
\begin{theorem}[Theorem 1.2 in \cite{Bates-Chatterje:MR4089496} and Theorem 1.2 in \cite{bakhtin-seo2020localization}]\label{thm:strongDisorderImpliesLocalization}
The following holds under either \Cref{continuousEnvironmentAssumptions} or \Cref{assumptions2}. If $\freeEnergy < \logExp$, then for all $\delta>0$ there are $K,\theta>0$ such that $\Prb$-almost surely localization with parameters $(\delta,K,\theta)$ holds for the sequence of endpoint measures $(\rho_0^n)_{n\in\N}.$
\end{theorem}
\Cref{thm:strongDisorderImpliesLocalization} in fact holds in arbitrary dimension and requires weaker conditions than \Cref{continuousEnvironmentAssumptions} and \Cref{assumptions2}. 
In Section~\ref{veryStrongDisorderSection} we derive \eqref{veryStrongDisorder_lattice}, and hence single point localization of the endpoint measures, in the continuous setting,
thus checking that~\Cref{onePointLocalizationReq} of \Cref{thm:generalThm} holds.
Specifically, we require the setting of \Cref{continuousEnvironmentAssumptions} and some additional assumptions on the density $\stepdensity,$ described below.
\begin{assumption}\label{densityAssumptions}
   The measure $\lambda$ is absolutely continuous with respect to Lebesgue measure. There is  a version of its density $\stepdensity$ that satisfies the following conditions.
    \begin{assumpRequirements}
        \item $\int_\R |x|^\momentnumber \stepdensity(x)dx < \infty$ for some $\momentnumber>2$.\label{momentsReq}
        \item $\sup_{x\in \R}p(x) < \infty.$\label{densityBoundedReq}
        \item There are $L,R\in\R$ satisfying $L< R$ such that $\stepdensity$ is nondecreasing on $(-\infty,L]$ and nonincreasing on $[R,\infty)$. In addition, $\stepdensity$ is bounded away from zero on $[L,R]$.
        \label{densityMonotonicityReq}
    \end{assumpRequirements}
\end{assumption}
\Cref{densityAssumptions} defines a broad class of densities. 
We prove the following lemma at the beginning of Section~\ref{logConcaveSection}.
\begin{lemma}\label{logConcaveSatisfiesAssumptions}
    If $\stepdensity$ is log-concave, then the measure $\lambda(dx) = \stepdensity(x)dx$ satisfies \Cref{densityAssumptions}.
\end{lemma}

\begin{theorem}\label{thm:veryStrongDisorderTheorem}
    Suppose the environment satisfies \Cref{continuousEnvironmentAssumptions} and the density $\stepdensity$ satisfies \Cref{densityAssumptions}. Then, 
    the very strong disorder property  \eqref{veryStrongDisorder_lattice} holds. 
\end{theorem}
We prove \Cref{thm:veryStrongDisorderTheorem} in Section~\ref{veryStrongDisorderSection}. \Cref{thm:veryStrongDisorderTheorem} and \Cref{thm:strongDisorderImpliesLocalization} immediately imply the following one-point localization result in one dimension.
\begin{theorem}\label{thm:onePointLocalization}
    Suppose the environment satisfies \Cref{continuousEnvironmentAssumptions} and $\stepdensity$ satisfies \Cref{densityAssumptions}. Then, for any $\delta > 0$ there are $K,\theta > 0$ such that $\Prb$-almost surely localization with parameters $(\delta,K,\theta)$ holds for the sequence of endpoint measures $(\rho_0^n)_{n\in \N}.$
\end{theorem}

The remainder of the paper is organized as follows: In Section~\ref{almostSureFinitenessSection}, we prove that partition functions are finite and that  
the polymer measures are well-defined. In Section~\ref{generalPolymersSection}, we prove \Cref{thm:generalThm},
a general result on joint localization.   In Section~\ref{continuousSpaceSection}, we apply this general result in the continuous setting and prove \Cref{thm:logConcaveCorollary} and \Cref{thm:mainThm1}.
In Section~\ref{assumptions2Section}, we apply it in the lattice setting
and prove \Cref{thm:mainThm2}. In Section~\ref{veryStrongDisorderSection}, we establish very strong disorder for continuous polymer models proving 
\Cref{thm:veryStrongDisorderTheorem} and hence \Cref{thm:onePointLocalization}. Section~\ref{veryStrongDisorderSection} is independent of Sections~\ref{generalPolymersSection}--\ref{assumptions2Section}.

\section{Finiteness of the Partition Functions}\label{almostSureFinitenessSection}
 The goal of this section is to prove the following theorem on finiteness of partition functions.
 
 \begin{theorem}\label{almostSureFinitenessCorollary}
 Relations \eqref{almostSureOFinite} and  \eqref{AlmostSurep2pFiniteness} hold under 
    \Cref{assumptions2}. They also hold under the combination of \Cref{continuousEnvironmentAssumptions} and \Cref{densityAssumptions}, 
\end{theorem}

We note that the conditions of the theorem are not the most general. For example, if $\lambda$ has bounded support, then all partition functions we consider are finite for every
realization of  $F\in\Omega$. 

The part of the theorem concerning \Cref{assumptions2} is trivial because in this case partition functions are finite sums of a.s.-finite r.v.'s. For the continuous case 
we will need the following auxiliary lemma.
\begin{lemma}\label{OfiniteExistence}
Let $\mathcal{G}$ be a countable set of random nonnegative continuous functions defined on the probability space $\Omega$. Suppose that for all $G\in \mathcal{G},$
\begin{equation}
    \sup_{x\in \R}\E[G(x)]<\infty.\label{boundedExpectation}
\end{equation}
Also suppose that there is $h\in L^1(\R,dx)$ and $a_0>0$ such that for all $a\in[0,a_0)$,
    \begin{equation}
        |\stepdensity(x-a) - \stepdensity(x)| \le h(x),\quad\text{Lebesgue-a.e.\ } x\in \R.\label{shiftenDensityDominated}
    \end{equation}
    Then, the set
    \begin{equation}
       \Omega' := \bigcap_{G\in \mathcal{G}}\Big\{\int_{\R}\stepdensity(x-a)G(x)dx < \infty \text{ for all }a\in \R\Big\}\label{finiteOSetG}
    \end{equation} 
    satisies  $\Prb(\Omega')=1$.
\end{lemma}

\begin{proof}
    Let us define a measurable set
    \begin{equation}
        \Omega'' = \bigcap_{G\in \mathcal{G}}\bigcap_{k\in\Z}\Big\{\int_\R\stepdensity(x-ka_0)G(x)dx <\infty\Big\}\cap \Big\{\int_\R h(x-ka_0) G(x) dx < \infty\Big\}.\label{FinitePartitionFunctionEvent}
    \end{equation}
  Condition \eqref{boundedExpectation} implies
    \[\E \int_\R \stepdensity(x-ka_0)G(x)dx<\infty\] and 
    \[\E \int_{\R}h(x-a)G(x)dx<\infty.\] 
    Therefore, $\Prb(\Omega'') = 1.$ We will show that $\Omega''\subset \Omega'$, establishing \eqref{finiteOSetG}.
    
    Let $G\in \mathcal{G}.$ Let $x_0\in \R$, $k\in \Z$, and $a\in [0,a_0)$ satisfy $x_0 = k a_0+a.$ If $F\in \Omega'',$ then
    \begin{align*}
        & \int_{\R}\stepdensity(x-x_0)G(x)dx \\
        &=  \int_{\R}\stepdensity(x-ka_0 - a)G(x)dx-  \int_{\R}\stepdensity(x-ka_0)G(x)dx +  \int_{\R}\stepdensity(x-k a_0)G(x)dx\\
        & \le  \int_\R |\stepdensity(x - ka_0 - a) - \stepdensity(x-ka_0)|G(x)dx + \int_{\R}\stepdensity(x-k a_0)G(x)dx\\
        & \le\int_\R h(x-ka_0)G(x)dx + \int_{\R}\stepdensity(x-k a_0)G(x)dx\\
        & <\infty.
    \end{align*}
    It follows that $\Omega''\subset \Omega'.$
\end{proof}

\bpf[Proof of Theorem~\ref{almostSureFinitenessCorollary}]
\Cref{exponentialMomentAssumption} implies
\begin{equation}
\label{eq:exp-moment1}
\E[e^{-F_0(0)}]<\infty,
\end{equation}
so $\E[Z_a^{m,n}] < \infty$ for all $a\in \X$ by Fubini's theorem and thus $\Prb\{Z_{a}^{m,n} < \infty\} = 1$ for all $a\in \R.$ To prove   \eqref{almostSureOFinite},
we need to extend this to uncountable intersection over all $a$. 
Since
 \[
 Z_a^{m,n} = e^{-F_m(a)}\int_\R \stepdensity(x-a)Z_x^{m+1,n}dx,
 \] 
 relation \eqref{almostSureOFinite} will follow from  Lemma~\ref{OfiniteExistence} applied to the collection
\[
\mathcal{G}=\Big\{x\mapsto Z_{x}^{m+1,n}:\quad  m,n\in \Z,\ m<n \Big\}
\]
once we check conditions \eqref{boundedExpectation} and \eqref{shiftenDensityDominated}. Condition~\eqref{boundedExpectation} follows from \eqref{eq:exp-moment1}.
Condition \eqref{shiftenDensityDominated} with some function $h\in L^1$ is satisfied if $x\mapsto \sup_{a\in (0,a_0)}\stepdensity(x-a)$ is integrable on $\R$. Since the latter follows
from \Cref{densityAssumptions}, the proof of \eqref{almostSureOFinite} under  the combination of \Cref{continuousEnvironmentAssumptions} and \Cref{densityAssumptions} is complete.

\medskip

Let us prove \eqref{AlmostSurep2pFiniteness}.

If $F\in \Ofinite$, then \eqref{p2l_p2pIntegration} implies that for all $m,n\in \Z$ satisfying $m<n,$ all $a\in \R$, and $\gamma$--a.e.\ $u\in \R$, 
\begin{equation}
    Z_{a,u}^{m,n}<\infty.\label{p2pFiniteness_ae}
\end{equation}
To prove~\eqref{AlmostSurep2pFiniteness}, we need to show that this relation holds for all $a,u\in \R$ (rather than just $\gamma$-a.e.~$u$), $\Prb$-almost surely. We can apply Lemma~\ref{OfiniteExistence} to
\[
\mathcal{G}=\Big\{x\mapsto \int_\R Z_{x,y}^{m+1,n-1}e^{-F_{n-1}(y)}dy:\quad  m,n\in\Z,\ m<n-1\Big\}
\]
because \eqref{boundedExpectation}
  follows from~\eqref{eq:exp-moment1} and we already checked condition~\eqref{shiftenDensityDominated} in our proof of~\eqref{almostSureOFinite}.  Thus we have $\Prb(\Omega') =1$, where  $\Omega'$ is defined in \eqref{finiteOSetG}.

Let $F\in\Omega'$. For every $a,u\in \R$ and every $m,n\in \Z$ satisfying $m<n-1,$  \Cref{densityBoundedReq} and the definition of $\Omega'$ imply
\begin{align}
    Z_{a,u}^{m,n} &= \int_\R Z_{a,y}^{m,n-1}e^{-F_{n-1}(y)}\stepdensity(u-y)dy \notag\\
    &\le \sup_{z\in \R}\stepdensity(z)\int_\R Z_{a,y}^{m,n-1}e^{-F_{n-1}(y)}dy \notag\\
    &\le e^{-F_m(a)}\sup_{z\in \R}\stepdensity(z)\int_\R\int_\R  \stepdensity(x-a)Z_{x,y}^{m+1,n-1} e^{-F_{n-1}(y)}dxdy \notag\\
    &< \infty.\label{p2pFiniteness_e}
\end{align}
If $m = n-1$, then for every $F\in \Omega,$ 
\[Z_{a,u}^{n-1,n} = e^{-F_{n-1}(a)}\stepdensity(u-a) < \infty.\] 
Therefore, $\Omega'\subset \OfinitePoint$ implying \eqref{AlmostSurep2pFiniteness} and completing the proof of Theorem~\ref{almostSureFinitenessCorollary}.
 \epf

\section{Joint Localization for General Polymer Models}\label{generalPolymersSection}
The main goal of this section is to prove \Cref{thm:generalThm}.
The main idea of the proof is that for $m<0$, the endpoint distribution $\rho_{0}^{m,n}(dy)$ is a mixture of endpoint distributions $\rho_{x}^{0,n}$ for
$x\in\R$. Thus, if localization holds for $(\rho_{0}^{m,n})_{n>m}$, i.e., the measures ~$\rho_{0}^{m,n}$, are mostly concentrated in a compact region $R$, at least some measures~$\rho_{x}^{0,n}$ also must assign large mass to $R$. Finally, a monotonicity argument allows us to conclude.
The endpoint distributions $(\rho^{m,n}_x(dy))_{x\in \X}$ of point-to-line polymers can be viewed as probability kernels from $\X$ to $\R$ equipped with Borel $\sigma$-algebras.
We recall that for measurable spaces $(S,\mathcal{S})$, $(E,\mathcal{E})$, a function $\kappa:S\times \mathcal{E}\to [0,1]$ is a probability kernel (from $S$ to $E$) if $\kappa_x$ is a probability measure on $(E,\mathcal{E})$ for every $x\in S,$ and $x\mapsto \kappa_{x}(A)$ is a measurable function for every $A\in \mathcal{E}.$ 
If $\chi$ is a probability measure on~$S$, then the convolution $\chi \kappa$ is the probability measure on $(E,\mathcal{E})$ given by
\[\chi \kappa(A) = \int_S\chi(dx) \kappa_x(A),\quad A\in \mathcal{E}.\]

For $d\in \N$ and $x,y\in \R^d$ we write $x\preceq y$ to mean $x_k \le y_k$ for all $k=1,\dots, d.$ We will say that a~function $f:\R^d\to \R$ is {\it coordinatewise nondecreasing} or simply {\it monotone} (for brevity) if $f(x) \le f(y)$ for all $x,y\in \R^d$ satisfying $x\preceq y.$ 

For measures $\nu,\mu$ on $\R^d,$ we say that $\nu$ stochastically dominates $\mu,$ and write $\mu \preceq \nu,$ if 
\begin{equation}
    \int_{\R^d}f(x)\mu(dx) \le \int_{\R^d}f(x)\nu(dx)\label{stochasticDominanceDef}
\end{equation}
for all monotone $f:\R^d\to [0,1].$  Equivalently, one can require~\eqref{stochasticDominanceDef} to hold for all monotone
$f:\R^d\to [0,\infty)$ or  for all monotone $f:\R^d\to \{0,1\}$. In $d=1$, $\mu \preceq \nu$ is equivalent to $\mu((-\infty,x])\ge \nu((-\infty,x])$ for every $x\in \R.$

\subsection{Deterministic Joint Localization}

We will derive \Cref{thm:generalThm} from the following general result.

\begin{proposition}\label{thm:abstractLocalization}
Let $a<b$ and $G\subset \X$ be a measurable set. Assume that $(\kappa^n)_{n\in \N}$ is a sequence of probability kernels from $\X$ to $\X$ and $(\chi^n)_{n\in \N}$ is a sequence of probability measures on $\X$. Suppose the following conditions hold:
\begin{requirements}
\item \label{localizationReq} (localization) For every $\delta >0,$ there are $K,\theta > 0$ such that 
localization with
parameters $(\delta,K,\theta)$ holds for the sequence of probability measures $(\chi^n \kappa^n)_{n\in \N}$.
\item\label{monotonicityReq} (monotonicity) $\kappa_x^n\preceq \kappa^n_y$ for every $x,y\in G\subset \X$ satisfying $x\le y$.
\item \label{positivityReq} (positive mass) 
\begin{align}
    \eta_- &:= \liminf_{n\to \infty} \chi^n((-\infty,a)\cap G)>0,\label{etaNeg_Def}\\
    \eta_+ &:= \liminf_{n\to \infty} \chi^n((b,\infty)\cap G)>0.\label{etaPos_Def}
\end{align}
\end{requirements}
Then, for every $\delta > 0,$ there are $K,\theta > 0$ such that  joint 
localization with parameters $(\delta,K,\theta)$
holds for the sequence $(\kappa_x^n)_{n\in \N,\,x\in (a,b)\cap G}$.
\end{proposition}

\begin{proof}
 Let $\delta > 0$ and define $\delta' = \frac{\delta}{3}\min(\eta_-, \eta_+).$ By \Cref{localizationReq} we can find $K,\theta > 0$ such that 
 localization with parameters  $(\delta', K,\theta)$ holds for 
 $(\chi^n\kappa^n)_{n\in\N}$. 

Suppose that for some $n\in \N$ and some $z\in \R$ we have
\begin{equation}
\chi^n\kappa^n(B_K(z)) > 1 - \delta'.\label{deltaPrimeN}
\end{equation}
Define $R  = \{x\in G\,:\,\kappa_x^n(B_K(z)) >  1-\delta\}$.  
We will show that sets $R\cap(-\infty,a) $ and $R\cap(b,\infty)$ are not empty  by showing that 
$\chi^n$ assigns positive mass to each of them.
 By \eqref{deltaPrimeN},
\begin{align*}
1 - \delta' &< \int_\X \kappa_x^n(B_K(z)) \chi^n(dx) \\
& \le \chi^n\left(\X\setminus G \right) + (1-\delta)\chi^n\left(G\setminus R\right) + \chi^n\left(R\right)\\
& =1 - \delta \chi^n\left(G\setminus R\right),
\end{align*}
so
\begin{equation}
    \chi^n\left(G\setminus R\right) < \frac{\delta'}{\delta}.\label{mnR_UpperBound}
\end{equation}

The definition of localization does not change if we discard finitely many values of~$n$, and so, due to \eqref{etaNeg_Def} and \eqref{etaPos_Def}, without loss of generality we will assume that $n$ is large enough to ensure that $\chi^n((-\infty,a)\cap G) >  \eta_-/2$ and $\chi^n((b,\infty)\cap G) > \eta_+/2$. By \eqref{deltaPrimeN} and \eqref{mnR_UpperBound}, 
\begin{align*}
    1-\delta' &< \int_\X \kappa_x^n(B_K(z)) \chi^n(dx)\\
    & \le \chi^n\left(R\cap (-\infty,a)\right) + (1-\delta) \chi^n\left(G\setminus R\right) + \chi^n\left(G\cap [a,\infty)\right) + \chi^n\left(\X\setminus G \right)\\
    & < \chi^n\left(R\cap (-\infty,a)\right) + (1-\delta) \frac{\delta'}{\delta} + \chi^n\left(G \right) - \frac{\eta_-}{2}+ \chi^n\left( \X\setminus G\right)\\
    & =  \chi^n\left(R\cap (-\infty,a)\right) + (1-\delta) \frac{\delta'}{\delta} + 1-\frac{\eta_-}{2}.
\end{align*}
Therefore,
\begin{equation*}
    \chi^n(R\cap (-\infty,a)) > \frac{\eta_-}{2} - \frac{\delta'}{\delta} > 0,
\end{equation*}
where the second inequality follows from the definition of $\delta'.$ Thus, $R\cap (-\infty,a) \neq \emptyset$, and there is $x_-\in G$ with $x_- < a$ such that 
\begin{equation}
\label{eq:x_-}
\kappa_{x_-}^n(B_K(z)) > 1-\delta.
\end{equation}
A similar proof shows that  there is $x_+\in G$ with $x_+ >  b$ such that 
\begin{equation}
\label{eq:x_+}
\kappa_{x+}^n(B_K(z)) > 1 - \delta.
\end{equation}
If $y\in (a,b)\cap G,$ then  $x_-<y<x_+$, so   \eqref{eq:x_-}, \eqref{eq:x_+} and \Cref{monotonicityReq} imply
\begin{equation*}
    \kappa_y^n((-\infty,z-K)) \le \kappa_{x_-}^n((-\infty,z-K)) \le  1- \kappa_{x_-}^n(B_K(z))<    \delta\label{kappan_gea_UpperBound}
\end{equation*}
and
\begin{equation*}
    \kappa_y^n((z+K,\infty)) \le \kappa_{x_+}^n((z+K,\infty)) \le 1- \kappa_{x_+}^n(B_K(z))< \delta.\label{kappan_leb_UpperBound}
\end{equation*}
The last two displays imply
\begin{equation}\label{kappan_forally_UpperBound}
\kappa_y^n(B_K(z)) > 1-2\delta,\, \quad \forall y\in (a,b)\cap G.
\end{equation}
In summary, for sufficiently large $n\in \N$, if  \eqref{deltaPrimeN} holds, then so does \eqref{kappan_forally_UpperBound}. Therefore,
\begin{align*}
    \liminf_{n\to \infty}\frac{1}{n}\sum_{i = 1}^n &\1\Big\{\sup_{z\in \R}\inf_{x\in (a,b)\cap G}\kappa_x^i(B_K(z)) > 1 - 2\delta\Big\} \\
    &= \liminf_{n\to \infty}\frac{1}{n}\sum_{i = 1}^n \1\Big\{\exists z \text{ s.t. } \forall x\in (a,b)\cap G,\,\kappa_x^i(B_K(z)) > 1 - 2\delta\Big\} \\
    &\ge \liminf_{n\to \infty}\frac{1}{n}\sum_{i = 0}^{n-1}\1\Big\{\sup_{z\in \R} \chi^i\kappa^i(B_K(z))> 1 - \delta'\Big\}\\
    & \ge \theta.
    \end{align*}
This completes the proof since $\delta>0$ is arbitrary.
\end{proof}

\subsection{Proof of \Cref{thm:generalThm}}

Let us check that the assumptions of \Cref{thm:abstractLocalization} hold true for the polymer endpoint distributions. For ease of reading, the dictionary between \Cref{thm:abstractLocalization} and \Cref{thm:generalThm} should be understood as follows. The collection of measures $(\kappa_x^n)_{x\in \X,\,n\in \N}$ will correspond to the endpoint measures $(\rho_x^n)_{x\in \X,\,n\in \N}.$ The sequence $(\chi^n)_{n\in \N}$ will correspond to the marginals $(\mu_{0}^{m,n} \pi_0^{-1})_{n=m,m+1,\dots}$.

\Cref{monotonicityReq} of \Cref{thm:abstractLocalization} is implied by the following lemma. Note that it requires neither \Cref{continuousEnvironmentAssumptions} nor \ref{assumptions2}. 
\begin{lemma}\label{thm:endpointMonotonicity}
Suppose that $F\in \OfiniteBoth$. Also suppose that \eqref{eq:Z-log-concavity} holds for some $n\in \N$, some $a$ and $b$ satisfying $a\le b$, and $\gamma^{\otimes 2}$-a.e.\ $(u,v)$ satisfying $u\le v$. Then $\rho_a^n \preceq \rho_b^n$. 
\end{lemma}
\begin{proof}
Fix $y\in \R.$ Using \eqref{eq:Z-log-concavity} and the symmetry of the set $\{(u,v):\ u\le y, v\le y\}$, we have
\begin{align*}
   \int_{u\le y}&Z_{a,u}^n \gamma(du) \int_{v\in\R}  Z_{b,v}^n\gamma(dv) \\ &=\int_{u\le y}\int_{v\in\R} Z_{a,u}^n Z_{b,v}^n\gamma(du)\gamma(dv) \\
    & = \int_{u\le y}\int_{v\le y} Z_{a,u}^n Z_{b,v}^n\gamma(du)\gamma(dv) + \int_{u\le y}\int_{v> y} Z_{a,u}^n Z_{b,v}^n\gamma(du)\gamma(dv)\\
    & \ge \int_{u\le y}\int_{v\le y} Z_{a,v}^n Z_{b,u}^n\gamma(du)\gamma(dv) + \int_{u\le y}\int_{v> y} Z_{a,v}^n Z_{b,u}^n\gamma(du)\gamma(dv)\\
    & = \int_{u\le y}\int_{v\in\R} Z_{a,v}^n Z_{b,u}^n\gamma(du)\gamma(dv)
    \\ 
    &= \int_{u\le y} Z_{b,u}^n \gamma(du) \int_{v\in\R}  Z_{a,v}^n\gamma(dv).
    \end{align*}
Dividing both sides by the finite, nonzero number $ \int_{v\in\R}  Z_{b,v}^n\gamma(dv)\int_{v\in\R}  Z_{a,v}^n\gamma(dv)$, we obtain
\begin{equation*}
    \rho_a^n((-\infty,y])=\frac{\int_{(-\infty,y]} Z_{a,u}^n \gamma(du)}{\int_{\R}Z_{a,v}^{n}\gamma(dv)} \ge \frac{\int_{(-\infty,y]} Z_{b,u}^n \gamma(du)}{\int_{\R}Z_{b,v}^{n}\gamma(dv)}
    =\rho_b^n((-\infty,y]),
\end{equation*}
completing the proof. 
\end{proof}

The following lemma shows that the polymer endpoint distributions are transition kernels as claimed and gives a disintegration formula.

\begin{lemma}\label{thm:endpointDisintegration} Let $m<k < n$. 
For every $F\in \OfiniteBoth$,
\[\rho_{0}^{m,n}(A) = \int_\R \rho_x^{k,n}(A) (\mu_{0}^{m,n}\pi_{k}^{-1})(dx) = (\mu_{0}^{m,n}\pi_{k}^{-1}) \rho_\cdot^{k,n} (A),\quad A\in \Bc.\]
\end{lemma}
\begin{proof}

For any $B\in\Bc$, 
\begin{align*}
(\mu_{0}^{m,n}\pi_{k}^{-1})(B)  = &\frac{1}{Z_0^{m,n}}\int_{\R^{k-m-1}\times B\times \R^{n-k-1}}e^{-\sum_{\ell = m}^{n-1}F_\ell(x_\ell)} \mathrm{P}_0^{m,n}(dx_m,\ldots, dx_{n})\\
= &\frac{1}{Z_0^{m,n}} \int_B Z_{0,y}^{m,k}Z_{y}^{k,n}\gamma(dy),
\end{align*}
which may be abbreviated to
\begin{align*}
(\mu_{0}^{m,n}\pi_{k}^{-1})(dy)  
= &\frac{1}{Z_0^{m,n}} Z_{0,y}^{m,k}Z_{y}^{k,n}\gamma(dy).
\end{align*}
Using this along with \eqref{eq:end-density} and Fubini's theorem we obtain, for any $A\in\Bc$:
\begin{align*}
\rho_0^{m,n}(A) & = \frac{1}{Z_{0}^{m,n}}\int_A Z_{0,y}^{m,n}\gamma(dy)\\
& =  \frac{1}{Z_{0}^{m,n}}\int_A \int_\R Z_{0,x}^{m,k} Z_{x,y}^{k,n}\gamma(dx)\gamma(dy)\\
& = \frac{1}{Z_{0}^{m,n}}\int_\R Z_{0,x}^{m,k} Z_{x}^{k,n}\int_A  \frac{Z_{x,y}^{k,n}}{Z_{x}^{k,n}}\gamma(dy)  \gamma(dx)\\
& = \frac{1}{Z_{0}^{m,n}}\int_\R Z_{0,x}^{m,k} Z_{x}^{k,n}\rho_x^{k,n}(A)\gamma(dx)\\
& = \int_\R \rho_x^{k,n}(A)(\mu_{0}^{m,n}\pi_{k}^{-1})(dx),
\end{align*}
which completes the proof.
\end{proof}

Let us now derive \Cref{thm:generalThm} from \Cref{thm:abstractLocalization}.
\begin{proof}[Proof of \Cref{thm:generalThm}]

We are going to apply \Cref{thm:abstractLocalization} to $\kappa_x^n = \rho_x^n$ 
and $\chi^n = \mu_{0}^{m,n} \pi_0^{-1},$ with an appropriately chosen $m$.  \Cref{thm:endpointDisintegration} implies that $\rho_{0}^{m,n} = \chi^n\kappa^n.$ \Cref{onePointLocalizationReq} of \Cref{thm:generalThm} means that for any $\delta > 0$ there are $K,\theta>0$ such that  localization with parameters $(\delta,K,\theta)$ holds for $(\rho_0^{m,n})_{n\ge 0}$. Thus, \Cref{localizationReq} of \Cref{thm:abstractLocalization} is verified. \Cref{monotonicityReq} of \Cref{thm:abstractLocalization} follows from \Cref{partitionRatioReq} of \Cref{thm:generalThm} combined with \Cref{thm:endpointMonotonicity}.

\Cref{positivityReq} of \Cref{thm:abstractLocalization} holds if we (i) choose $a=-r$ and $b=r$, where
$r>0$  is chosen to ensure 
$A\subset [-r,r]$ and (ii) use \Cref{positivityMarginalReq} of \Cref{thm:generalThm} to find $m>0$ such that \eqref{eq:mass_below_minus} 
and~\eqref{eq:mass_above_plus} hold.
\end{proof}

\section{Continuous Space Joint Localization}\label{continuousSpaceSection}
In this section, we prove \Cref{thm:logConcaveCorollary} and \Cref{thm:mainThm1} using \Cref{thm:generalThm}. The standing setting in the rest of this section is that of \Cref{continuousEnvironmentAssumptions} and in addition we always assume that $\stepdensity$ is log-concave with energy $V$ given by \eqref{potentialDef}.

\subsection{Proof of \Cref{thm:logConcaveCorollary}}\label{logConcaveSection}
We need to check that the conditions of \Cref{thm:generalThm} hold $\Prb$-almost surely for $G = \R.$ Due to \Cref{almostSureFinitenessCorollary},  we may restrict ourselves to the event $\OfiniteBoth$.

Since $p$ is log-concave, the set 
\begin{equation}
\label{eq:support_of_p}
E= \{x\in \R\,:\, \stepdensity(x) > 0\}
\end{equation}
is an interval. Adjusting the values of $\stepdensity$ at the endpoints of $E$ if needed, we will always assume that $\stepdensity$ is continuous on $\overline{E}$, the closure of $E$. Clearly, $Z_{a,u}^n > 0$ if and only if $u-a\in nE=\{nx\,:\, x\in E\}.$

First we check \Cref{onePointLocalizationReq} of \Cref{thm:generalThm}. Due to \Cref{thm:onePointLocalization}, this amounts to proving \Cref{logConcaveSatisfiesAssumptions}.

\begin{proof}[Proof of \Cref{logConcaveSatisfiesAssumptions}]
By Lemma 1 in \cite{10.1214/09-EJS505}, $\stepdensity$ has all finite moments and in particular satisfies \Cref{momentsReq}.

\Cref{densityBoundedReq} holds true because $V$ is a convex function satisfying $V(x) > -\infty$ for all $x\in \R$ and $\lim_{|x|\to\infty}V(x)=+\infty$.

Let us check \Cref{densityMonotonicityReq}. As $\stepdensity$ is continuous on $\overline{E},$ there is $x_0\in \overline{E}$ such that $\stepdensity$ attains its (positive) maximum at $x_0.$ Equivalently, $V$ attains its (finite) minimum at $x_0.$ There exist $L,R\in \overline{E}$ such that $x_0\in [L,R]\subset \overline{E}$ and $\sup_{x\in [L,R]}V(x) < \infty.$ Indeed, if $x_0$ is in the interior of~$E$, pick $L,R$ so that $[L,R]$ is in the interior 
of~$E$ and contains~$x_0$. If $x_0\in\{\inf E,\sup E\}$, we can 
use $V(x_0) < \infty$ to
choose a sufficiently small segment $[L,R]$ with one of the endpoints coinciding with $x_0$
 to ensure $\sup_{x\in [L,R]}V(x) < \infty$.

It follows that $\stepdensity$ is strictly positive on $[L,R].$ The convexity of $V$ implies that it is nonincreasing on $(-\infty,L]$ and nondecreasing on $[R,\infty),$  so $V = -\log \stepdensity$ is nondecreasing on $(-\infty,L]$ and nonincreasing on $[R,\infty).$ 
\end{proof}

Now we establish \Cref{partitionRatioReq} of \Cref{thm:generalThm}. First we require a lemma proving a monotonicity property of polymer measures with log-concave step distributions.

\begin{lemma}\label{stochasticDominanceLogConcave}
    For all $F\in \OfiniteBoth,$ $a\in\R$ and all $u,v\in a + nE$ satisfying $u\le v$, we have $\mu^n_{a,u}\preceq \mu^{n}_{a,v}.$
\end{lemma}
Our proof is an extension of the argument used for Lemma 7.3 in \cite{Bakhtin-Li:MR3911894}, a specific case of our lemma, with $\stepdensity(x) = \frac{1}{\sqrt{2\pi}}e^{-x^2/2}$. 
It relies on the following fact.
\begin{lemma}\label{stochasticDominance_for_Convolutions}
       Let $\nu$ be a positive measure on $\R$ and let 
    \begin{equation}
        A=\bigg\{z\in\R: \ 0 < \int_{\R} \stepdensity (z-x)\nu(dx) <\infty\bigg\}\label{nonDegenerateNuCondition}.
    \end{equation}
Then, for all $z,z'\in A$ satisfying $z\le z'$, we have
    \begin{equation}
        \frac{\int_{(-\infty,y]} \stepdensity (z-x)\nu(dx)}{\int_{\R} \stepdensity (z-x)\nu(dx)} \ge \frac{\int_{(-\infty,y]} \stepdensity (z'-x)\nu(dx)}{\int_{\R} \stepdensity (z'-x)\nu(dx)}.\label{monotoneMarginalResult}
    \end{equation}
\end{lemma}
We derive \Cref{stochasticDominanceLogConcave} from \Cref{stochasticDominance_for_Convolutions} first. 

\begin{proof}[Proof of \Cref{stochasticDominanceLogConcave} ] We prove the lemma only for $a = 0$ since the proof is almost identical for arbitrary $a$. Let $\pi_{k,n}$ denote the projection of a path $(x_0,x_1,\dots)$ to the coordinates $k$ through $n$, an element of $\R^{n-k+1}.$ We will prove
    \begin{equation}
        \mu_{0,z}^n\pi_{k,n}^{-1} \preceq \mu_{0,z'}^n\pi_{k,n}^{-1},\qquad  \forall z,z'\in nE,\,z\le z',\label{inductiveClaim}
    \end{equation}
    for all $k \in \{n,n-1,\dots, 0\}$ using induction. The statement of the lemma is \eqref{inductiveClaim} with $k = 0$. Note that~\eqref{inductiveClaim} with $k=n$ is trivially true because $\mu_{0,x}^n\pi_{n,n}^{-1} = \delta_x$ for all $x\in nE$, and $\delta_z\preceq \delta_{z'}$ for all $z,z'\in \R$ with $z\le z'$.

  Suppose that \eqref{inductiveClaim} holds for some $k\in \{n,\dots,1\}.$ Let $f:\R^{n-k+2}\to [0,1]$  be a bounded monotone function. 
  
 Denoting  $x=(x_{k-1},\dots,x_n)\in \R^{n-k+2}$,
  disintegrating $\mu_{0,z}^n\pi_{k-1,n}^{-1}$,  and using Fubini's theorem, we obtain 
    \begin{align}
    \notag
        &\int_{\R^{n-k+2}} f(x)\mu_{0,z}^n\pi_{k-1,n}^{-1}(dx)  \\
        & = \int_{\R^{n-k+2}} f(x_{k-1},\dots, x_n)\mu_{0,z}^n\pi_{k,n}^{-1}(dx_k,\dots,dx_n) \mu_{0,x_k}^k\pi_{k-1}^{-1}(dx_{k-1})
        \label{kMinus1Disintegration}
        \\
        \notag
        &=\int_{\R^{n-k+1}} \bar{f}(x_k,\dots,x_n) \mu_{0,z}^n\pi_{k,n}^{-1}(dx_k,\dots,dx_n),
    \end{align}
    where 
    \begin{equation}
        \bar{f}(x_k,\dots,x_n) = \int_\R f(x_{k-1},\dots, x_n)\mu_{0,x_k}^k\pi_{k-1}^{-1}(dx_{k-1}).\label{defOfBarFGood}
    \end{equation}
    It will be convenient to redefine $\bar{f}$ on the set of $(x_{k},\dots, x_n)$ for which $\mu_{0,x_k}^k$ is the zero measure.
    Let 
    \[S = \{(x_k,\dots, x_n)\in \R^{n-k+1}\,:\,x_k\in kE\}.\]
    For $x\in S$ define $\bar{f}(x)$ by \eqref{defOfBarFGood}. For $x\notin S$ define
    \begin{equation}
        \bar{f}(x) = 
        \begin{cases}
            \sup\{\bar{f}(y)\,:\,y\in S,\,y\preceq x\}, & \exists y\in S\text{ s.t.\ }y\preceq x\\
            0,& \text{otherwise.}
        \end{cases}\label{defOfBarFBad}
    \end{equation}
    If $\bar{f}$ is monotone on $S$ then it is easy to check using \eqref{defOfBarFBad} that $\bar{f}$ is monotone on all of $\R^{n-k+1}.$

    We will now show that $\bar{f}$ is monotone on $S$. For every $x,x'\in S$ with $x\preceq x',$
    \begin{equation}
        \int_\R f(x_{k-1},x_k,\dots, x_n)\mu_{0,x_k}^k\pi_{k-1}^{-1}(dx_{k-1})\le \int_\R f(x_{k-1},x_k',\dots, x_n')\mu_{0,x_k}^k\pi_{k-1}^{-1}(dx_{k-1})\label{integrandNondecreasing}
    \end{equation} 
   due to monotonicity of $f$. Now let 
    \[\nu(dx) = Z_{0,x}^{0,k-1}dx.\]
    Then, for all $w\in kE$, 
    \begin{equation}
    \label{eq:proj_on_k-1}
    \mu_{0,w}^k\pi_{k-1}^{-1}(dx) = \frac{\stepdensity(w-x)\nu(dx)}{\int_\R \stepdensity(w-x')\nu(dx')}.
    \end{equation}
    Since $F\in \OfiniteBoth$, \eqref{eq:proj_on_k-1} defines a well-defined probability measure for all $w\in kE.$
     \Cref{stochasticDominance_for_Convolutions} and \eqref{eq:proj_on_k-1} imply that for all $w,w'\in kE$ satisfying $w\le w'$,
    \begin{equation}
        \mu_{0,w}^k\pi_{k-1}^{-1} \preceq \mu_{0,w'}^k\pi_{k-1}^{-1}.\label{kMinus1MarginalMonotonicity}
    \end{equation}
    Since for every $(x_k,x_{k+1},\dots, x_n)\in \R^{n-k+1}$ the map $x_{k-1}\mapsto f(x_{k-1},x_k,\dots, x_n)$ is monotone, \eqref{kMinus1MarginalMonotonicity} implies that for all $x'\in \R^{n-k+1}$ and all $x_k\le x_k'$,
    \begin{equation}
        \int_\R f(x_{k-1},x_k',\dots, x_n')\mu_{0,x_k}^k\pi_{k-1}^{-1}(dx_{k-1}) \le \int_\R f(x_{k-1},x_k',\dots, x_n')\mu_{0,x_k'}^k\pi_{k-1}^{-1}(dx_{k-1}).\label{onePointMonotonicity}
    \end{equation}
Inequalities \eqref{integrandNondecreasing} and \eqref{onePointMonotonicity} imply that for all $x,x'\in S$ with $x\preceq x',$
\begin{align}
    \bar{f}(x) &= \int_\R f(x_{k-1},x_k,\dots, x_n)\mu_{0,x_k}^k\pi_{k-1}^{-1}(dx_{k-1})\notag\\
    &\le \int_\R f(x_{k-1},x_k',\dots, x_n')\mu_{0,x_k'}^k\pi_{k-1}^{-1}(dx_{k-1})\notag \\
    &= \bar{f}(x'),\notag
\end{align}
so $\bar{f}$ is monotone on $S$.

The inductive assumption that \eqref{inductiveClaim} holds for $k$, equality \eqref{kMinus1Disintegration}, and the fact that~$\bar{f}$ is monotone imply that the left-hand side of \eqref{kMinus1Disintegration} is monotone  in $z\in nE.$ As this statement holds for every monotone function $f:\R^{n-k+2}\to [0,1]$, we conclude that \eqref{inductiveClaim} holds true for $k-1$, which completes the induction step. The proof of 
\Cref{stochasticDominanceLogConcave} will be complete once we  prove \Cref{stochasticDominance_for_Convolutions}. 
\end{proof}

 \begin{proof}[Proof of \Cref{stochasticDominance_for_Convolutions}] 
Let us  first show that if $x,x',z,z'\in \R$ satisfy $x\le x'$ and $z\le z'$, then
    \begin{equation}
        \stepdensity (z-x)\stepdensity (z'-x') \ge \stepdensity (z'-x)\stepdensity (z-x').\label{stepDensityPathCrossing}
    \end{equation}
We can assume $\stepdensity(z'-x) > 0$, because otherwise \eqref{stepDensityPathCrossing} is trivially satisfied. In this case, $z'-x\in E.$ If $\stepdensity(z-x) = 0$, then $z-x\notin E$ and thus, since $E$ is convex, we must also have $z-x'\notin E$, due to our assumptions on $x,x',z,z'$. This implies $\stepdensity(z-x') = 0$,  so \eqref{stepDensityPathCrossing} is satisfied. We now consider the case where $\stepdensity(z'-x) > 0$ and $\stepdensity(z-x) > 0$, i.e.\ $V(z'-x),V(z-x)<\infty$, where $V$ is the convex function defined in \eqref{potentialDef}.

 We have $z-x'\le z-x$, $x'-x\ge 0,$ and $z'-z\ge 0.$ Since the difference quotient of a convex function is an increasing function, we have 
    \begin{align}
        \frac{V(z-x) - V(z-x')}{x'-x} &\le \frac{V(z-x + (z'-z)) - V(z-x'+(z'-z))}{x'-x}\notag \\
        & = \frac{V(z'-x) - V(z'-x')}{x'-x}.\label{differenceQuotientInequality}
    \end{align}
    Multiplying both sides of \eqref{differenceQuotientInequality} by $x'-x$ and rearranging, we obtain 
    \[-V(z-x) - V(z'-x') \ge -V(z'-x) - V(z-x').\]
    Taking the exponential of both sides, we obtain~\eqref{stepDensityPathCrossing}. 

    Now we can use \eqref{stepDensityPathCrossing} to  write
    \begin{align*}
        \int_{(-\infty,y]}& \stepdensity (z-x) \nu(dx)\int_{\R}\stepdensity (z'-x') \nu(dx')\\
        & =\int_{x\le y}\int_{x'\in \R} \stepdensity (z-x)\stepdensity (z'-x') \nu(dx')\nu(dx) \\
        & = \int_{x\le y}\int_{x'\le y} \stepdensity (z-x)\stepdensity (z'-x') \nu(dx')\nu(dx) + \int_{x\le y}\int_{x'> y} \stepdensity (z-x)\stepdensity (z'-x') \nu(dx')\nu(dx)\\
        & \ge \int_{x\le y}\int_{x'\le y} \stepdensity (z-x')\stepdensity (z'-x) \nu(dx')\nu(dx) + \int_{x\le y}\int_{x'> y} \stepdensity (z'-x)\stepdensity (z- x') \nu(dx')\nu(dx)\\
        & = \int_{x\le y}\int_{x'\in \R} \stepdensity (z'-x)\stepdensity (z-x') \nu(dx')\nu(dx)\\
        & = \int_{(-\infty,y]}\stepdensity (z'-x) \nu(dx)\int_{\R}\stepdensity (z-x') \nu(dx').
    \end{align*}
   Dividing both sides by $\int_{\R}\stepdensity (z'-x') \nu(dx')\int_{\R}\stepdensity (z-x') \nu(dx')$ we obtain  \eqref{monotoneMarginalResult}.  
\end{proof}

\begin{proof}[Proof of \Cref{partitionRatioReq} of \Cref{thm:generalThm}] 
We assume that $F\in \OfiniteBoth$. If $Z_{b,u}^n=0$, then \eqref{eq:Z-log-concavity} holds trivially, so let us assume that $Z_{b,u}^{n} > 0$, or, equivalently, recalling the definition of $E$ in \eqref{eq:support_of_p}, $u-b \in nE$.

First, consider the case $Z_{b,v}^n=0$, i.e., $v-b\notin nE$. Since $v-a > v-b$ and $n E$ is convex, it follows that $v-a\notin nE$ and so $Z_{a,v}^n = 0.$ As a result, if $Z_{b,v}^n=0$ and $Z_{b,u}^{n} > 0$ then $Z_{a,v}^n = 0$ and so \eqref{eq:Z-log-concavity} is satisfied. 

We now consider the case $Z_{b,v}^n> 0.$ By disintegrating the point-to-point partition function in the first coordinate we have
\begin{align}
    \frac{Z_{a,u}^n}{Z_{b,u}^n} &= \frac{e^{-F_0(a)}\int_\R \stepdensity (x-a)Z_{x,u}^{1,n} dx}{Z_{b,u}^n}\notag\\
    & = \frac{e^{-F_0(a)+F_0(b)}\int_\R \frac{\stepdensity (x-a)}{\stepdensity (x-b)}\stepdensity (x-b)Z_{x,u}^{1,n}e^{-F_0(b)}dx}{Z_{b,u}^n}\notag\\
    & = e^{-F_0(a)+F_0(b)}\int_\R \frac{\stepdensity (x-a)}{\stepdensity (x-b)}\mu_{b,u}^n\pi_1^{-1}(dx).\label{ratioOfPartitionFunctionTrick}
\end{align}
Note that for $\mu_{b,u}^n\pi_1^{-1}$-a.e.~$x$, we have $p(x-b) >0$,  hence the right-hand side of~\eqref{ratioOfPartitionFunctionTrick} is well-defined. Equation \eqref{stepDensityPathCrossing} shows that for $a,b\in \R$ satisfying $a\le b$, the map $y\mapsto \frac{\stepdensity (y-a)}{\stepdensity (y-b)}$ is nonincreasing on the set $\{y\in \R\,:\,\stepdensity(y-b)> 0\}$, which has full measure under $\mu_{b,u}^n\pi_1^{-1}$ and $\mu_{b,v}^n\pi_1^{-1}$. Since $Z_{b,u}^n> 0$ and $Z_{b,v}^n>0$, \Cref{stochasticDominanceLogConcave} implies that $\mu_{b,u}^n\preceq \mu_{b,v}^n.$ Therefore,
\begin{equation*}
    e^{-F_0(a)+F_0(b)}\int_\R \frac{\stepdensity (x-a)}{\stepdensity (x-b)}\mu_{b,u}^n\pi_1^{-1}(dx) \ge e^{-F_0(a)+F_0(b)}\int_\R \frac{\stepdensity (x-a)}{\stepdensity (x-b)}\mu_{b,v}^n\pi_1^{-1}(dx).
\end{equation*} 
Under the assumption $Z_{b,v}^n> 0$, the same computation as in \eqref{ratioOfPartitionFunctionTrick} shows that the right-hand side
 is $\frac{Z_{a,v}^n}{Z_{b,v}^n}$, and the lemma follows.
\end{proof}

\Cref{thm:logConcaveCorollary} then follows because we have proven that if the environment satisfies \Cref{continuousEnvironmentAssumptions} and $\stepdensity$ is log-concave, then \Cref{onePointLocalizationReq,partitionRatioReq} of \Cref{thm:generalThm} hold $\Prb$-almost surely.

\subsection{Proof of \Cref{thm:mainThm1}}\label{sec:proof_under_A}
In this section, we prove \Cref{thm:mainThm1} on joint localization for GRW polymers under \Cref{continuousEnvironmentAssumptions}.
In this case, $\stepdensity$ is Gaussian, hence log-concave, so to 
apply \Cref{thm:logConcaveCorollary} it remains to prove that \Cref{positivityMarginalReq} of \Cref{thm:generalThm} holds $\Prb$-almost surely.

We need to know that the marginals of finite-dimensional point-to-line polymer measures have uniformly positive density with respect to the random walk reference measure. For GRW polymers under \Cref{continuousEnvironmentAssumptions}, one can actually prove a stronger result,
convergence of densities to a positive limit, uniform on compact sets. Expressing densities of marginals of polymer measures via partition functions, one can see that convergence of those densities is tightly related to convergence 
of ratios of certain partition functions. The latter convergence is a direct corollary of  Theorem 3.2  of \cite{Bakhtin-Li:MR3911894}, one of the main results of  
that paper describing the basins of attraction for global solutions of the Burgers equation with random kick forcing:

\begin{theorem}[\cite{Bakhtin-Li:MR3911894}] 
\label{thm:partitionRatiosGaussian}
Let $p=g$ under \Cref{continuousEnvironmentAssumptions}.
Then, with $\Prb$-probability~$1$,
 for every $x,y\in \R,$ the sequence $(Z_x^n / Z_y^n)_{n\in \N}$ converges uniformly on compact sets to a $C^1(\R\times \R)$ function $v(x,y)$. The function $v$ satisfies
\begin{align*}
0  < v(x,y)&  < \infty,\\
\log v(x,y)& = o(|x|),\quad x\to \pm\infty,\\
\log v(x,y)& = o(|y|),\quad y\to \pm \infty.
\end{align*}
\end{theorem}

In fact, the conditions under which \Cref{thm:positiveMassGaussian} holds are weaker than \Cref{continuousEnvironmentAssumptions} and in particular positive correlation is not required.

For our purposes it is more convenient to use an intermediate result  
proved in~\cite{Bakhtin-Li:MR3911894} as a part of the argument for Theorem 3.2:
\begin{theorem}[\cite{Bakhtin-Li:MR3911894}]\label{thm:positiveMassGaussian}
Let $p=g$ under \Cref{continuousEnvironmentAssumptions}. Then  there is a probability one event $\bar{\Omega}\subset \Omega$ such that for every $F\in \bar{\Omega}$ the following holds. For any $m\in \Z$, $x\in \R,$ there is a measure $\mu_x^{m,\infty}$ on the space of paths $\gamma:\{m,m+1\dots,\}\to \R$ such that for any $k> m,$ the sequence $(\mu_x^{m,n} \pi_k^{-1})_{n=m}^\infty$ converges in distribution to $\mu_x^{m,\infty} \pi_k^{-1}.$ Further, the marginals of $\mu_x^{m,\infty}$ are absolutely continuous with respect to Lebesgue measure, with everywhere positive density.
\end{theorem}

\begin{proof}[Proof of \Cref{thm:mainThm1}]
Using \Cref{thm:logConcaveCorollary} we need only verify that \Cref{positivityMarginalReq} of \Cref{thm:generalThm} holds $\Prb$-almost surely, with $G= \R.$ \Cref{thm:positiveMassGaussian} shows that 
\[\liminf_{n\to \infty}(\mu_x^{-1,n} \pi_0^{-1})(U) > 0\] 
for any set $U$ with positive Lebesgue measure, and so in particular $\Prb$-almost surely \Cref{positivityMarginalReq} holds for any $r > 0$ with $m = -1$, 
\end{proof}

\section{Joint Localization Under \Cref{assumptions2}}\label{assumptions2Section}

In this section, we prove \Cref{thm:mainThm2} and hence establish joint localization for the simple random walk model with one-step measure $\lambda = \frac{1}{2}\delta_{-1}+\frac{1}{2}\delta_{1}$.
This model is equivalent to the up-right path model of \cite{Janjigian--Rassoul-Agha:MR4089495} obtained from ours by a coordinate change (rotation by $\pi/2$ and scaling by $\sqrt{2}$).
Throughout this section we work under Assumption~\ref{assumptions2} designed to ensure that the assumptions in \cite{Janjigian--Rassoul-Agha:MR4089495} are satisfied.

Our goal is to check that all three requirements of \Cref{thm:generalThm} hold for this model for any bounded subset $A$ of $2\Z$.

\bpf[Proof of \Cref{onePointLocalizationReq} of \Cref{thm:generalThm}]
We recall that for any non-constant environment in dimension one, the very strong disorder property
\[\lim_{n\to \infty}\frac{1}{n}\log Z^n < \log \E e^{- F_0(0)}\]
was shown to hold true when the collection $(F_0(x))_{x\in \Z}$ is i.i.d. and $F_0(0)$ has all exponential moments in \cite{CV06}. The same result under \Cref{assumptions2} was shown to be true in \cite{LW09}. It was shown in
\cite{Bates-Chatterje:MR4089496} (one can also apply generalizations in~\cite{bakhtin-seo2020localization}, \cite{Bates:MR4269204}) that this property (i.e., the discrepancy between the annealed and quenched average free energies) implies localization, thus ensuring  \Cref{onePointLocalizationReq}. \epf

To see that \Cref{partitionRatioReq} of \Cref{thm:generalThm} holds, we will use the following lemma which is only a restatement of Lemma C.3 of \cite{Janjigian--Rassoul-Agha:MR4089495} for our
coordinate system in a convenient form.

For $n\in\N$, we will need 
\begin{equation}
\label{eq:rangeV}
 V_n = \{-n,-n+2,\dots, n-2,n\}, 
 \end{equation}
 the set of points accessible by the simple random walk at time $n$.
 
 \begin{lemma}[\cite{Janjigian--Rassoul-Agha:MR4089495}]\label{thm:ratioMonotonicityDiscrete}
Let $F\in\Omega$, $a\in 2\Z$, 
 $u,v\in \Z$, $u\le v$, and $n\in\N$.
 If
 \begin{equation}
 \label{eq:u,v_in_cone1}
u,v\in a+1+ V_{n-1},
\end{equation}
  then
 \begin{equation}\label{eq:ratioDiscrete1}
 \frac{Z_{a+1,v}^{1,n}}{Z_{a,v}^{0,n}} \ge \frac{Z_{a+1,u}^{1,n}}{Z_{a,u}^{0,n}}.
  \end{equation}
If 
 \begin{equation}
 \label{eq:u,v_in_cone2}
u,v\in a-1+ V_{n-1},
\end{equation}
  then
 \begin{equation}\label{eq:ratioDiscrete2}
 \frac{Z_{a-1,v}^{1,n}}{Z_{a,v}^{0,n}} \le \frac{Z_{a-1,u}^{1,n}}{Z_{a,u}^{0,n}}.
 \end{equation}
\end{lemma}
\bpf[Proof of \Cref{partitionRatioReq} of \Cref{thm:generalThm}] Let us fix $n\in\N$.
The claim is obvious if $a=b$ or $u=v$, so let us assume $a,b\in 2\Z$  and $u,v\in\Z$ satisfy $a<b$ and $u<v$. 
Assuming
\begin{equation}
\label{eq:Z_positive}
Z^{0,n}_{a,v}, Z^{0,n}_{b,u}>0,
\end{equation}
 we obtain
 $v\in a+V_n$ and $u\in b+V_n$ and see
 that for all $x\in\{a+1,a+3,\ldots,b-1\}$, $u,v\in x+V_{n-1}$. This allows to check conditions \eqref{eq:u,v_in_cone1}, \eqref{eq:u,v_in_cone2} and 
 apply inequalities~\eqref{eq:ratioDiscrete1},~\eqref{eq:ratioDiscrete2} to these intermediate values.
 In particular, for every $y\in\{a,a+2,\ldots, b-2\}$, we obtain
\[\frac{Z_{y,u}^{0,n}}{Z_{y,v}^{0,n}} \ge \frac{Z_{y+1,u}^{1,n}}{Z_{y+1,v}^{1,n}} \ge  \frac{Z_{y+2,u}^{0,n}}{Z_{y+2,v}^{0,n}}.\]
Combining these inequalities over all these values of $y$, we obtain
\begin{equation}
\label{eq:ratioDiscrete_ab}
\frac{Z_{a,u}^n}{Z_{a,v}^n} \ge \frac{Z_{b,u}^n}{Z_{b,v}^n},
\end{equation}
which is equivalent to \eqref{eq:Z-log-concavity} under our assumption~\eqref{eq:Z_positive}.

It remains to consider the case where~\eqref{eq:Z_positive} is violated.
If $Z^{0,n}_{a,v}=0$, then \eqref{eq:Z-log-concavity} obviously holds. Also, if $Z^{0,n}_{b,v}=0$, then $Z^{0,n}_{a,v}=0$ and thus \eqref{eq:Z-log-concavity} holds.
This completes the proof of
\Cref{partitionRatioReq} for an arbitrary $A\subset 2\Z$.
\epf

Finally, we establish \Cref{positivityMarginalReq} using boundedness of point-to-point partition function ratios. The following lemma is a simple corollary of Theorem 4.16 in \cite{Janjigian--Rassoul-Agha:MR4089495}.
\begin{lemma}\label{thm:partitionRatiosDiscrete}
For $\Prb$-almost every $F\in \Omega,$ every $v\in (-1,1),$ and every sequence $(x_n)_{n\in \N}$ such that $\lim_{n\to \infty}\frac{x_n}{n} = v,$ and $a,b\in 2\Z,$
\begin{equation}\label{eq:partitionsRatiosDiscrete}
0 < \liminf_{n\to \infty}\frac{Z_{a,x_n}^n}{Z_{b,x_n}^n} \le \limsup_{n\to \infty}\frac{Z_{a,x_n}^n}{Z_{b,x_n}^n} < \infty.
\end{equation}
\end{lemma}
\begin{proof}
Theorem 4.16 of \cite{Janjigian--Rassoul-Agha:MR4089495} shows that 
\begin{equation}\label{eq:simpleRatioBound1}
0 < \liminf_{n\to \infty}\frac{Z_{a+1,x_n}^{1,n}}{Z_{a,x_n}^{0,n}} \le \limsup_{n\to \infty}\frac{Z_{a+1,x_n}^{1,n}}{Z_{a,x_n}^{0,n}} < \infty
\end{equation}
and
\begin{equation}\label{eq:simpleRatioBound2}
0 < \liminf_{n\to \infty}\frac{Z_{a-1,x_n}^{1,n}}{Z_{a,x_n}^{0,n}} \le \limsup_{n\to \infty}\frac{Z_{a-1,x_n}^{1,n}}{Z_{a,x_n}^{0,n}} < \infty
\end{equation}
$\Prb$-almost surely for every $a\in \Z$. If $a\le b$, then we can write $\frac{Z_{a,x_n}^n}{Z_{b,x_n}^n}$ as the telescoping product
\[\frac{Z_{a,x_n}^n}{Z_{b,x_n}^n} = \prod_{\ell = a}^{b-1} \frac{Z_{\ell,x_n}^{e_\ell,n}}{Z_{\ell+1,x_n}^{e_{\ell+1},n}}\]
where $e_a = 0$ and $e_\ell = 1 - e_{\ell-1}$ for $\ell=a+1,\ldots,b$. We can then apply the displays~\eqref{eq:simpleRatioBound1} and \eqref{eq:simpleRatioBound2} to obtain the result.
\end{proof}

The limiting behavior of ratios of partition functions is tightly connected to the properties of the  shape function.  
As discussed in \cite{Janjigian--Rassoul-Agha:MR4089495}, \cite{Rassoul-Agha--Seppalainen:MR3176363}, and \cite{Comets:MR3444835}, 
there is a convex, continuous, deterministic, even function  $\Lambda:[-1,1] \to \R$, called the \textit{shape function}, such that
\begin{equation}
    \lim_{n\to \infty} \max_{x\in a+V_n}\left|\frac{1}{n}\log Z_{a,x}^n - \Lambda((x-a)/n)\right| = 0\label{shapeTheorem}
\end{equation}
$\Prb$-almost surely.  In addition, $\Lambda$ is not constant if the environment is not deterministic.

\begin{remark}
Theorem 4.16 of \cite{Janjigian--Rassoul-Agha:MR4089495} gives a stronger result than what we used in the proof of \Cref{thm:partitionRatiosDiscrete}.
It gives upper and lower bounds for \eqref{eq:simpleRatioBound1} and \eqref{eq:simpleRatioBound2} in terms of Busemann functions. In addition, Theorem 3.8 in \cite{Janjigian--Rassoul-Agha:MR4089495} implies that if the shape function  $\Lambda$ is differentiable on $(-1,1)$ then for every $v\in (-1,1)$ there is a probability one event such that the ratios in \eqref{eq:partitionsRatiosDiscrete} converge on it. 
\end{remark}

Convergence of ratios of point-to-line partition functions to finite positive random variables (exponentials of Busemann functions) follows immediately from Theorem~3.8 in \cite{Janjigian--Rassoul-Agha:MR4089495} under the assumption that the shape function is differentiable everywhere. 
It is widely believed that this differentiability assumption holds for a broad class of potentials $F$. 

Although no direct analogue of \Cref{thm:partitionRatiosGaussian} is available for lattice polymers,
we are still able to prove the following useful result without any differentiability assumptions:
\begin{lemma}\label{thm:partitionRatioDiscrete}
For $\Prb$-almost every $F\in \Omega$ and for every $a,b\in 2\Z,$ 
\begin{equation}
    0 < \liminf_{n\to \infty} \frac{Z_{a}^{n}}{Z_{b}^{n}} \le \limsup_{n\to \infty} \frac{Z_{a}^{n}}{Z_{b}^{n}} < \infty.\label{eq:partitionRatioDiscrete}
\end{equation}
\end{lemma}
\begin{proof}
We will prove that \eqref{eq:partitionRatioDiscrete} holds on the probability one event that \eqref{shapeTheorem} and the conclusion of \Cref{thm:partitionRatiosDiscrete} hold.

Suppose first that $b > a.$ Recalling the definition of $V_n$ in~\eqref{eq:rangeV}, for any $v\in (-1,1)$, we can write
\begin{align*}
Z_{a}^n & = \sum_{x\in a + V_n} Z_{a,x}^n \\
& = \sum_{\substack{x\in a+V_n,\\x < v n + a}} Z_{a,x}^n +  \sum_{\substack{x\in a+V_n,\\x \ge v n + a}} Z_{a,x}^n\\
& \le  (\Sigma_1(v,n)+\Sigma_2(v,n)) Z_{b}^n ,
\end{align*}
where 
\begin{align*}
\Sigma_1(v,n)&= \frac{\sum_{\substack{x\in a+V_n,\\x < v n+a}} Z_{a,x}^n}{Z_{b}^n},\\
\Sigma_2(v,n)& =\max_{\substack{x\in a+V_n,\\x\ge v n + a}}\left\{\frac{Z_{a,x}^n}{Z_{b,x}^n} \right\}.
\end{align*}
We note that if $x\in a+V_n$ and $x\ge v n + a$, then $x\in b+V_n$, so $Z_{b,x}^n>0$ and $\Sigma_2(v,n)$ are well-defined. Let us  show that 
\begin{equation}
\label{eq:limsup_Sigma_2}
\limsup_{n\to\infty}\Sigma_2(v,n)<\infty,\quad \forall v\in(-1,1).
\end{equation}

Let $x_n(v) =\min\{x\in a + V_n : x\ge v n + a\}$.
It follows from \eqref{eq:ratioDiscrete_ab} (or, equivalently from \Cref{partitionRatioReq} of \Cref{thm:generalThm}) that for all $x \ge x_n(v)$
\[\frac{Z_{a,x}^n}{Z_{b,x}^n} \le \frac{Z_{a,x_n(v)}^n}{Z_{b,x_n(v)}^n},\]
so 
\[
\Sigma_{2}(v,n)\le \frac{Z_{a,x_n(v)}^n}{Z_{b,x_n(v)}^n}.\]
Since $\lim_{n\to \infty}\frac{x_n(v)}{n} = v,$ we can apply \Cref{thm:partitionRatiosDiscrete} to the right-hand side of the above and obtain
\[\limsup_{n\to \infty} \Sigma_{2}(v,n)  \le\limsup_{n\to \infty} \frac{Z_{a,x_n(v)}^n}{Z_{b,x_n(v)}^n} < \infty,\]
proving \eqref{eq:limsup_Sigma_2}. 

Our next goal is to find $v^*\in(-1,1)$ such that 
\begin{equation}
\label{eq:limsup_Sigma_1}
\limsup_{n\to\infty}\Sigma_1(v^*,n)<\infty.
\end{equation}
This estimate along with \eqref{eq:limsup_Sigma_2} applied to $v=v^*$ implies the upper bound in the lemma for the case $b>a$.

Since $\Lambda$ is concave and not constant, there is $v^*\in(-1,1)$ such that \[\max_{w\in [-1,1]}\Lambda(w) > \Lambda(v^*)\] and $\Lambda(v) \le \Lambda(v^*)$ for all $-1\le v < v^*.$ Then, by \eqref{shapeTheorem} and the choice of $v^*$,
\begin{equation}
    \sum_{\substack{x\in a+V_n,\\x < v^* n + a}} Z_{a,x}^n = \sum_{\substack{x\in a+V_n,\\x < v^* n + a}} e^{n\Lambda((x-a)/n) + o(n)} \le (n+1)e^{n\Lambda(v^*) + o(n)}.\label{vStarUpperBound}
\end{equation}
Here $o(n)$ is uniform over $x$ in the summation.

Consider a maximizer $u^*\in (-1,1)$ of $\Lambda.$ There is $\epsilon > 0$ such that 
\begin{equation}
\label{eq:lambda_star}
\lambda^* := \inf_{|w- u^*|<\epsilon} \Lambda(w) > \Lambda(v^*).
\end{equation}
For sufficiently large $n$, we have
\begin{equation}
    Z_{b}^n \ge  \sum_{\substack{x \in b + V_n,\\|(x-b)/n - u^*| < \epsilon}} Z_{b,x}^n \ge \frac{1}{2}\epsilon n e^{n\lambda^* + o(n)}.\label{uStarUpperBound}
\end{equation}
Inequalities \eqref{vStarUpperBound}, \eqref{eq:lambda_star}, and \eqref{uStarUpperBound} imply
\[
\Sigma_1(v^*,n)\le \frac{2(n+1)}{n\epsilon} e^{n(\Lambda(v^*) - \lambda^*) + o(n)}\to 0
\]
as $n\to \infty.$ Thus, \eqref{eq:limsup_Sigma_1} is established, which completes the proof of the upper bound for the case $b>a$. 

The proof of the upper bound for the case where $b < a$ is similar. For the lower bound it suffices to reverse the roles of $a$ and $b$ and apply the upper bound.
\end{proof}

With \Cref{thm:partitionRatioDiscrete} at hand, we can prove the following result.

\begin{proposition}\label{thm:positiveMassDiscrete}
If  $m < k < n$ and $y\in a+V_{k-m},$ then $\Prb$-almost surely,
\[\liminf_{n\to \infty} (\mu_a^{m,n} \pi_k^{-1})(\{y\}) > 0.\]
\end{proposition}
\begin{proof}
    For all $y\in a + V_{k-m}$, we have
    \[(\mu_a^{m,n} \pi_k^{-1})(\{y\}) = \frac{Z_{a,y}^{m,k}Z_{y}^{k,n}}{Z_{a}^{m,n}}.\] 
   The factor $Z_{a,y}^{m,k}>0$ on the right-hand side does not depend on $n$. The remaining factors satisfy
    \begin{align*}
    \frac{Z_{y}^{k,n}}{Z_{a}^{m,n}} & = Z_{y}^{k,n} \left( \sum_{z\in a+ V_{k-m}} Z_{a,z}^{m,k}Z_{z}^{k,n}\right)^{-1}
    = \left( \sum_{z\in a+V_{k-m}} Z_{a,z}^{m,k}\frac{Z_{z}^{k,n}}{Z_{y}^{k,n}}\right)^{-1}.
    \end{align*}
    From  \Cref{thm:partitionRatioDiscrete}, $\limsup_{n\to \infty}\frac{Z_{z}^{k,n}}{Z_{y}^{k,n}} < \infty$ for all $z\in a+V_{k-m}$.  The result follows because $a+V_{k-m}$ is a finite set.
    \end{proof}

\bpf[Proof of \Cref{positivityMarginalReq} of \Cref{thm:generalThm}] Let $m<0$ be an integer such that $V_{-m}\cap(r,\infty)$ is not empty. Then, for any $y$ in this set, \Cref{thm:positiveMassDiscrete} implies
    \[\liminf_{n\to \infty} (\mu_{0}^{m,n} \pi_0^{-1})((r,\infty)) \ge \liminf_{n\to \infty} (\mu_{0}^{m,n} \pi_0^{-1})(\{y\})>0, \]
    so \eqref{eq:mass_above_plus} holds. The argument for \eqref{eq:mass_below_minus} is similar. This completes the proof of \Cref{positivityMarginalReq}
    under \Cref{assumptions2}
   and the proof of  the entire \Cref{thm:mainThm2}.
    \epf    
    
\section{Very Strong Disorder in Continuous Space}\label{veryStrongDisorderSection}

\subsection{Derivation of the theorem from auxiliary results}
\label{sec:main_very_strong_disorder}

In this section, we establish \Cref{thm:veryStrongDisorderTheorem} in the setting of \Cref{continuousEnvironmentAssumptions} and \Cref{densityAssumptions}. First, we briefly discuss the existence of the limit in \eqref{freeEnergyDef}. The reasoning is almost exactly the same as in the discrete case, given for example in \cite{CSY03}.  We give a sketch of the argument here for completeness. It is easy to prove that  the sequence $(\frac{1}{n}\E[\log Z^n])_{n\in \N}$ is subadditive which
implies that the limit
\begin{equation}
\label{eq:def_free_energy}
\freeEnergy = \lim_{n\to \infty}\frac{1}{n}\E[\log Z^n] 
\end{equation}
exists. Then, one can use an exponential concentration inequality which easily adapts to our continuous space setting:
 \begin{theorem}[Theorem 1.4 in \cite{LW09}]
    Suppose $\E[e^{|F_0(0)|}]<\infty.$ Then, there is a constant $a>0$ such that 
    \[\Prb\left\{\frac{1}{n}\left|\log Z^n - \E[\log Z^n]\right|>t \right\}\le \begin{cases}
        2 e^{-nat^2},& 0\le t\le 1\\
        2 e^{-nat}, & t>1.
    \end{cases}\]
\end{theorem}
It follows that the sequence $(\frac{1}{n}\log Z^n)_{n\in\N}$ converges $\Prb$-a.s.\ and in $L^p(\Omega)$ for any $p\in [1,\infty)$ and
\begin{equation}
    \lim_{n\to \infty} \frac{1}{n}\log Z^n =  \freeEnergy.\notag
\end{equation}
Recall that $\logExp = \log\E[e^{-F_0(0)}]$. To prove \Cref{thm:veryStrongDisorderTheorem}, it suffices to prove
\begin{equation}
    \liminf_{n\to \infty}\frac{1}{n}\E[\log Z^n] < \logExp.\label{liminfReduction}
\end{equation}

Our proof of \eqref{liminfReduction} uses comparison of $\frac{1}{n}\E[\log Z^n]$ with fractional moments of~$Z^n$ and extends the ideas used in \cite{CV06} and \cite{CSY03} in the discrete case to the continuous case. There are some additional difficulties in the continuous setting. First, the authors in \cite{CV06} rely on the inequality 
\begin{equation}
    \left(\sum_{i} |x_i|\right)^\theta \le \sum_{i} |x_i|^\theta,\quad  x_i\in \R,\ i\in \Z,\ \theta\in(0,1). \label{thetaInequality}
\end{equation}
The version of this inequality where the sum is replaced by an integral does not hold in general. However, we are able to employ a discretization procedure that enables effective use of \eqref{thetaInequality}. Our discretization method is described in Section~\ref{fractionalMomentUpperBound}.

Second, the analysis in \cite{CSY03} uses the overlap metric
\begin{equation}
    I_n = (\rho^{ n})^{\otimes 2}\{X=Y\}\label{overlapMeasure_Discrete}
\end{equation}
where $X,Y\sim \rho^{n}$ are independent. Equivalently, \eqref{overlapMeasure_Discrete} is the $\ell^2(\Z)$ norm of the density of $\rho^n$. In the continuous setting, $I_n=0$ and so the same object is not useful. In addition, we found that the $L^2(\R)$ norm of the density did not have the same desirable properties as \eqref{overlapMeasure_Discrete}. Instead, we employ the new family of overlap measurements defined on probability measures $\nu$ on $\R$ by
\begin{equation}
    I(r, \nu) := \nu^{\otimes 2}\{|X-Y|<r\},\quad r>0.\label{overlapMeasure_Continuous}
\end{equation}
A similar but slightly different metric was used in \cite{Rovira-Tindel:MR2129770} in the continuous space and continuous time setting. In Section~\ref{decayOfFractionalMoments} we prove an important comparison property between $I(r,\nu)$ and $I(R,\nu)$ for $R\neq r$ and then use arguments from \cite{CSY03} to conclude.

Additionally, to simplify the presentation we will assume throughout this section that $\stepdensity$ has mean zero, i.e.\ $c:=\int_\R x \stepdensity(x)dx = 0.$ 
In fact, if $c\ne0$, then we can reduce this noncentered case to the centered one using the  stationarity of the environment:
\[Z^n = \mathrm{P}^m \left( e^{-\sum_{k = 0}^{m-1}F_k(S_k)}\right) \stackrel{d}{=} \tilde{Z}^n,\]
where $\tilde{Z}^n = \mathrm{P}^m \left( e^{-\sum_{k = 0}^{m-1}F_k(S_k-kc)}\right)$ is the partition function associated to the centered random walk density $\tilde{p}(x) = p(x - c)$,
and 
\[\frac{1}{n}\E[\log Z^n] = \frac{1}{n}\E[\log\tilde{Z}^n].\]

In the remainder of this section, we state several auxiliary lemmas and derive \Cref{thm:veryStrongDisorderTheorem} from them. The proofs of the auxiliary Lemmas are given in Sections~\ref{fractionalMomentUpperBound} and \ref{decayOfFractionalMoments}.
We must introduce some notation first. Let 
\begin{equation}
    W_{a,u}^{n,m} := Z_{a,u}^{n,m}e^{-\logExp(m-n)},\qquad  W^{m} := Z^{m} e^{-\logExp m}\label{normalizedPartitionFcnsDef}
\end{equation}
be the normalized point-to-point and point-to-line partition functions, respectively. Note that $(W^m)_{m\in \N}$ is a martingale, and in particular $\E[W^m] = 1$ for all $m.$ We define $W_{y,J}^{k,i} = \int_J W_{y,x}^{k,i}dx.$ For $\delta>0$, let 
\begin{equation}
    \mathcal{J}^\delta = \{J_k^\delta,\,k\in \Z\}, \label{defofScriptJ}
\end{equation}
where 
\begin{equation}
    J_k^\delta = [k\delta,(k+1)\delta).\label{defofJk}
\end{equation}
The collection $\mathcal{J}^\delta$ is a disjoint covering of $\R$ by intervals of length $\delta.$ We frequently omit the~$\delta$ superscript for brevity. In addition, constants denoted by $C$ or $C'$ may change line by line.

The following lemma can be seen as an approximate factorization of the partition function that allows us to relate the quenched free energy at step~$nm$ to a fractional moment at step $m$.
\begin{lemma}\label{factorizationLemma}
        For any $\delta > 0$, $\theta\in (0,1),$ and $m\in \N:$
        \[\frac{1}{nm}\E[\log W^{nm}] \le \frac{1}{\theta m}\log \sum_{J\in\mathcal{J}^\delta} \E\left[\left(\sup_{x\in J_0}W_{x,J}^{0,m}\right)^\theta \right].\]
\end{lemma}
The next result is an involved technical proposition that allows us to get rid of the supremum in \Cref{factorizationLemma} and replace it with evaluation of $W_{x,J_k}^{0,m}$ at $x=0.$
    \begin{lemma}\label{boundForSupW}
        Let $\varepsilon > 0.$ There is are numbers $\theta_0\in (0,1)$ and $\delta_0>0$ such that for all $\theta\in (\theta_0,1)$ and $\delta\in(0,\delta_0)$, there is a positive constant $C$ such that for all $m\in \N$ 
        \begin{equation}
            \sum_{J\in\mathcal{J}^\delta} \E\left[\left(\sup_{x\in J_0}W_{x,J}^{0,m}\right)^\theta \right] \le  C \sum_{J\in\mathcal{J}^\delta} \E\left[\left(W_{0,J}^{0,m}\right)^\theta \right] + C \delta^{-\frac{1}{2}-\varepsilon}m^{-\frac{1}{2}+\varepsilon} + C \delta^{\frac{3}{4}-\varepsilon}m^{\frac{7}{20}+\varepsilon}.\label{boundForSupWDisplay}
        \end{equation}
    \end{lemma}
    The next proposition replaces the sum of partial fractional moments on the right-hand side of \eqref{boundForSupWDisplay} with the fractional moment $\E[(W^m)^\theta]$.
\begin{lemma}\label{boundForW0}
    Let $\varepsilon>0$. There is a number $\theta_0\in (0,1)$ such that for all $\theta\in(\theta_0,1)$ there is a positive constant $C$ such that for all $m\in\N$ and $\delta\in (0,1)$,
    \begin{equation}
    \label{eq:boundForW0}
    \sum_{J\in\mathcal{J}^\delta} \E\left[\left(W_{0,J}^{0,m}\right)^\theta \right]\le  C \left(\delta^{-1+\frac{1}{2}\varepsilon} m^{\frac{1}{2}-\frac{1}{4}\varepsilon}(\E[(W^m)^\theta])^{1-\varepsilon}\right)^{\frac{1}{2-\varepsilon}}.
    \end{equation}
\end{lemma}

We will also need the following proposition giving a decay rate of the fractional moment $\E[(W^m)^{\theta}].$ Recall that $\momentnumber$ is the number of moments of $\lambda$ as stated 
in \Cref{momentsReq}.
\begin{lemma}\label{lem:fractionalMomentDecay}
    Let $\varepsilon>0.$ There is number $\theta_0\in(0,1)$ such that if $\theta\in (\theta_0,1)$ then there is a positive constant $C$ such that for all $n\in \N,$
    \[\E[(W^{n})^\theta] \le C n^{-\frac{1}{2} \momentnumber + \varepsilon}.\]
\end{lemma}
The proofs of Lemmas~\ref{factorizationLemma}--\ref{boundForW0} are given in Section~\ref{fractionalMomentUpperBound}. The proof of 
\Cref{lem:fractionalMomentDecay} is given in Section~\ref{decayOfFractionalMoments}.

\begin{proof}[Proof of \Cref{thm:veryStrongDisorderTheorem}]

Let $h \in (0,1)$ and $\varepsilon>0$, to  be chosen later. Let $\delta_m = m^{-h}$. Taking $\varepsilon < \frac{1}{2}\momentnumber - 1,$ we can use \Cref{lem:fractionalMomentDecay} to see that for $m$ sufficiently large,
\begin{equation}
    \E[(W^{m})^\theta] < m^{-1}.\label{fractionalMomentOneOverNBound}
\end{equation}

Lemmas~\ref{factorizationLemma}--\ref{boundForW0} and \eqref{fractionalMomentOneOverNBound} imply that there is a constant $\theta\in(0,1)$ and a positive constant $C$ such that for all $m$ sufficiently large,
\begin{align}
    \frac{1}{nm}&\E[\log W^{nm}]\notag\\
    & \le \frac{1}{\theta m}\log \left(C \left(\delta_m^{-1+\frac{1}{2}\varepsilon} m^{\frac{1}{2}-\frac{1}{4}\varepsilon}(\E[(W^m)^\theta])^{1-\varepsilon}\right)^{\frac{1}{2-\varepsilon}} + C \delta_m^{-\frac{1}{2}-\varepsilon}m^{-\frac{1}{2}+\varepsilon} + C \delta_m^{\frac{3}{4}-\varepsilon}m^{\frac{7}{20}+\varepsilon}\right)\notag\\
    & = \frac{1}{\theta m}\log \left(C \left(m^{h - \frac{1}{2}h\varepsilon+\frac{1}{2}-\frac{1}{4}\varepsilon}(\E[(W^m)^\theta])^{1-\varepsilon}\right)^{\frac{1}{2-\varepsilon}} + C m^{\frac{1}{2}h-\frac{1}{2}+ h\varepsilon+\varepsilon} + C m^{-\frac{3}{4}h+\frac{7}{20}+h\varepsilon+\varepsilon}\right).\label{freeEnergyFirstUpperBound}
    \\
  &\le  \frac{1}{\theta m}\log \left(C\left(m^{h - \frac{1}{2} -\frac{1}{2}h\varepsilon+\frac{3}{4}\varepsilon}\right)^{\frac{1}{2-\epsilon}} + C m^{\frac{1}{2}h-\frac{1}{2}+ h\varepsilon+\varepsilon} + C m^{-\frac{3}{4}h+\frac{7}{20}+h\varepsilon+\varepsilon}\right).\notag 
\end{align}
Choose $h\in (0,\frac{1}{2})$ close enough to $\frac{1}{2}$ such that $-\frac{3}{4}h + \frac{7}{20} < 0$. For instance $h = \frac{29}{60}$ suffices. Then, choose $\varepsilon > 0$ so that 
\[h - \frac{1}{2} - \frac{1}{2}h\varepsilon+\frac{3}{4}\varepsilon < 0,\]
\[\frac{1}{2}h-\frac{1}{2}+ h\varepsilon+\varepsilon < 0,\]
and
\[-\frac{3}{4}h+\frac{7}{20}+h\varepsilon+\varepsilon < 0.\]
Then all the powers of $m$ in \eqref{freeEnergyFirstUpperBound} are negative, so
\[\frac{1}{nm}\E[\log W^{nm}] \le \frac{1}{\theta m} \log \left(o_m(1) \right), \quad m\to \infty.\]
Taking $m$ large enough, we obtain that there is a constant $\Delta$ such that
\[\frac{1}{nm}\E[\log W^{nm}] < \Delta < 0,\quad n\in\N.\]
Combining this with  $\frac{1}{nm}\E[\log W^{nm}] = \frac{1}{nm}\E[\log Z^{nm}] - \logExp$, we obtain~\eqref{liminfReduction}  and complete the proof of the theorem.
\end{proof}

\subsection{Fractional Moment Upper Bounds for Free Energy}\label{fractionalMomentUpperBound}
In this section, we prove Lemmas~\ref{factorizationLemma}--\ref{boundForW0}.

We first prove a short lemma concerning the expectation of the normalized partition function. 
\begin{lemma}\label{thm:expectedNormalizedPartitionFcn}
    If $U\subset \R$ is a measurable set, $k,m\in\Z$, $k<m$, and $x\in \R,$ then 
    \begin{equation}
        \E[W_{x,U}^{k,m}] = \mathrm{P}^{k,m}_x\{S_m\in U\}.\label{expectedNormalizedPartitionFcn}
    \end{equation}
\end{lemma}
\begin{proof}
    We have 
    \begin{align}
        \E[W_{x,U}^{k,m}] & = \E \int_U W_{x,y}^{k,m}dy\notag\\
        & = \E \int_U \mathrm{P}_{x,y}^{k,m}\left(e^{-\sum_{i=k}^{m-1}F_i(S_i)} e^{- (m-k)\logExp}\right) dy\notag\\
        & = \int_U \mathrm{P}_{x,y}^{k,m}\left(\E\left[e^{-\sum_{i=k}^{m-1}F_i(S_i)} e^{- (m-k)\logExp} \right] \right)dy\notag\\
        & = \int_U \mathrm{P}_{x,y}^{k,m} (1)dy\notag\\
        & = \mathrm{P}_x^{k,m}\{S_m\in U\},\notag 
    \end{align}
    proving \eqref{expectedNormalizedPartitionFcn}.
\end{proof}

\begin{proof}[Proof of \Cref{factorizationLemma}]

    For $x_1,\dots, x_n\in \R$, we define 
    \[W_{x_1,\dots,x_n}^m :=W_{0,x_1}^{0,m}\cdots W_{x_{n-1},x_n}^{(n-1)m,nm}.\] 
    Then, we use Jensen's inequality and~\eqref{thetaInequality}:
\begin{align}\notag
    \frac{1}{nm}\E\log W^{nm} & = \frac{1}{nm}\E\log \sum_{J_{k_1},\dots,J_{k_n} \in \mathcal{J}^\delta} \int_{J_{k_1}\times \cdots \times J_{k_n}}W_{x_1,\dots,x_n}^m dx\\
    \label{eq:discretizing_W}
    & = \frac{1}{\theta nm}\E\log \bigg(\sum_{J_{k_1},\dots,J_{k_n}\in \mathcal{J}^\delta}\int_{J_{k_1}\times \cdots\times J_{k_n}}W_{x_1,\dots,x_n}^m dx\bigg)^\theta\\
    & \le \frac{1}{\theta n m}\log \sum_{J_{k_1},\dots,J_{k_n}\in \mathcal{J}^\delta}\E\bigg(\int_{J_{k_1}\times \cdots \times J_{k_n}}W_{x_1,\dots,x_n}^m dx\bigg)^\theta.\notag
\end{align}
We can use an inductive argument to obtain a product
estimate on each of the integrals on the right-hand side:
\begin{align*}
    \int_{J_{k_1}\times \cdots \times J_{k_n}}&W_{x_1,\dots,x_n}^m dx  = \int_{J_{k_1}\times \cdots \times J_{k_{n-1}}}W_{x_1,\dots,x_{n-1}}^{m-1}\int_{J_{k_n}}W_{x_{n-1},x_n}^{(n-1)m,nm}dx_{n} dx_1\dots dx_{n-1}\\
    & \le \int_{J_{k_1}\times \cdots\times J_{k_{n-1}}}W_{x_1,\dots,x_{n-1}}^{m-1} dx_1\dots dx_{n-1}\cdot \sup_{x_{n-1}\in J_{k_{n-1}}}\int_{J_{k_n}}W_{x_{n-1},x_n}^{(n-1)m,nm}dx_{n}\\
    & \le \dots\le \prod_{\ell=1}^n\sup_{y\in J_{k_{\ell-1}}}\int_{J_{k_\ell}} W_{y,x}^{(\ell-1)m,\ell m}dx.
\end{align*}
By the independence of the environment at different times and spatial stationarity,
\begin{align*}
    \E\bigg(\int_{J_{k_1}\times \cdots \times J_{k_n}}W_{x_1,\dots,x_n}^m dx\bigg)^\theta & \le\prod_{\ell=1}^n\E\bigg(\sup_{y\in J_{0}}W_{y,(J_\ell-k_{\ell-1}\delta)}^{0,m}\bigg)^\theta.
\end{align*}
So,
\begin{align*}
    \sum_{J_{k_1},\dots,J_{k_n}\in \mathcal{J}^\delta}\E\bigg(\int_{J_{k_1}\times \cdots \times J_{k_n}}W_{x_1,\dots,x_n}^m dx\bigg)^\theta \le \bigg( \sum_{J_k\in \mathcal{J}^\delta}\E\bigg(\sup_{y\in J_{0}}W_{y,J_k}^{0,m}\bigg)^\theta \bigg)^n.
\end{align*}
Using this estimate on the right-hand side of \eqref{eq:discretizing_W}, we complete the proof.
\end{proof}

\begin{proof}[Proof of \Cref{boundForSupW}]
    
First, note that 
\[\sup_{|x|<\delta}W_{x,J_k}^{0,m} \le \sup_{|x|<\delta} e^{- F_0(x)} \cdot \sup_{|x|<\delta} W_{x,J_k}^{0,m} e^{ F_0(x)}.\]
Also, $W_{x,J_k}^{0,m} e^{ F_0(x)}$ is in fact independent of $F_0.$ From this independence, we get 
\begin{align}
    \E\bigg(\sup_{|x|<\delta}W_{x,J_k}^{0,m}\bigg)^\theta& \le \E\bigg(\sup_{|x|<\delta}\{ e^{- F_0(x)}\} \cdot \sup_{|x|<\delta}\{W_{x,J_k}^{0,m} e^{ F_0(x)}\}\bigg)^\theta\notag\\
    & = \E\bigg(\sup_{|x|<\delta}e^{-\theta  F_0(x)}\bigg) \E\bigg(\sup_{|x|<\delta}W_{x,J_k}^{0,m} e^{ F_0(x)}\bigg)^\theta\notag\\
    &= \E\bigg(\sup_{|x|<\delta}e^{-\theta  F_0(x)}\bigg) \bigg[\E\bigg(e^{-\theta  F_0(0)}\bigg)\bigg]^{-1}\E\bigg(e^{-\theta  F_0(0)}\bigg)\E\bigg(\sup_{|x|<\delta}W_{x,J_k}^{0,m} e^{ F_0(x)}\bigg)^\theta\notag\\
    &\le C_0 \E\bigg(\sup_{|x|<\delta}W_{x,J_k}^{0,m} e^{ F_0(x)- F_0(0)}\bigg)^\theta,\label{C0firstUppberBound}
\end{align}
where
\[C_0 := \E\bigg(\sup_{|x|<1}e^{-\theta  F_0(x)}\bigg) \bigg[\E\bigg(e^{-\theta  F_0(0)}\bigg)\bigg]^{-1}.\]
Note that we assume here $\delta < 1$ so that the constant $C_0$ does not depend on $\delta.$ 

We define 
\begin{equation}
\label{defOfDensityRatio}
    \densityRatio = 2\frac{\sup_{x\in \R}\stepdensity(x)}{\inf_{x\in [L,R]}\stepdensity(x)},
\end{equation}
where $L$ and $R$ are as in \Cref{densityMonotonicityReq}. \Cref{densityMonotonicityReq,densityBoundedReq} imply that $\densityRatio < \infty.$
We also define  
\[A_\delta := \Big\{y\in \R\,:\, \sup_{|x|<\delta} p(y-x)\ge \densityRatio^2 p(y)\Big\}.\] 
 Using \eqref{thetaInequality},  \eqref{annealedBoundDef}, the independence of $W_{x,J_k}^{0,m} e^{ F_0(x)}$ from $F_0$, Jensen's inequality, bounding the integrand pointwise, and \Cref{thm:expectedNormalizedPartitionFcn} we obtain
\begin{align}
  \label{AdeltaBound}
    \E&\left(\sup_{|x|<\delta}\{W_{x,J_k}^{0,m}e^{ F_0(x)- F_0(0)}\}\right)^\theta = \E\left(\sup_{|x|<\delta}\{Z_{x,J_k}^{0,m}e^{ F_0(x)- F_0(0) - m\logExp}\}\right)^\theta \\
    & = \E\left(\sup_{|x|<\delta}\int_\R p(y-x) Z_{y,J_k}^{1,m} e^{- F_0(0)-\logExp m}dy\right)^\theta\notag\\
    & = \E\left(\sup_{|x|<\delta}\int_\R p(y-x) W_{y,J_k}^{1,m} e^{- F_0(0)-\logExp}dy\right)^\theta\notag\\
    & \le \E \left(\int_{A_\delta^c} \sup_{|x|<\delta} p(y-x) W_{y,J_k}^{1,m} e^{- F_0(0)-\logExp}dy \right)^\theta + \E \left(\int_{A_\delta} \sup_{|x|<\delta}p(y-x) W_{y,J_k}^{1,m} e^{- F_0(0)-\logExp}dy\right)^\theta\notag\\
    & \le \densityRatio^{2\theta}\E \left(\int_{A_{\delta}^c} p(y) W_{y,J_k}^{1,m} e^{- F_0(0)-\logExp}dy \right)^\theta + R_k^\theta\notag\\
    & \le \densityRatio^{2\theta} \E[(W_{0,J_k}^{0,m})^\theta] + R_k^\theta,\notag
\end{align}
where 
\begin{equation}
    R_k : = \int_{A_\delta} \sup_{|x|<\delta} p(y-x)\mathrm{P}_y^{1,m}\{S_m\in J_k\}dy.
\end{equation}

The lemma will follow once we find $\theta_0\in(0,1)$ such that for all $\theta\in(\theta_0,1),$
\begin{equation}
    \sum_{k\in \Z} R_k^\theta = O(\delta^{-\frac{1}{2}-\varepsilon}m^{-\frac{1}{2}+\varepsilon} + \delta^{\frac{3}{4}-\varepsilon}m^{\frac{7}{20}+\varepsilon}).\label{boundRkGoal}
\end{equation}
The lemma will follow from \eqref{boundRkGoal}, \eqref{AdeltaBound}, and \eqref{C0firstUppberBound}.

Let us take any $D\in\R$ and $\delta_0 > 0$ such that $(D-\delta_0,D+\delta_0)\subset (L,R)$. We claim that for $\delta\in (0,\delta_0)$,
\begin{align}
    \sup_{|x|<\delta}\stepdensity(y-x) \le \densityRatio \stepdensity(y-\delta),\qquad \forall y\ge D\label{supBUpperBound}\\
    \sup_{|x|<\delta} \stepdensity(y-x) \le \densityRatio \stepdensity(y+\delta),\qquad  \forall y\le D.\label{supBLowerBound}
\end{align}
We verify \eqref{supBUpperBound}, and the analysis for \eqref{supBLowerBound} is similar. Note that for $y\in [R+\delta,\infty),$
\[\sup_{|x|<\delta} \stepdensity(y-x) \le \stepdensity(y-\delta)\]
since $p$ is nonincreasing on $[R,\infty)$ by \Cref{densityMonotonicityReq}. Also, for all $y\in [D,R+\delta),$ since $\delta < \delta_0$ we have 
\[\stepdensity(y-\delta) \ge \inf_{z\in [L,R]}\stepdensity(z)\] and 
\[\sup_{|x|<\delta}\stepdensity(y-x) \le \sup_{z\in \R} \stepdensity(z).\] 
This implies~\eqref{supBUpperBound} since
\begin{align*}
    \sup_{|x|<\delta}\stepdensity(y-x) \le \frac{\sup_{x\in \R}\stepdensity(y-x) }{\inf_{x\in [L,R]} \stepdensity(y-x)} \stepdensity(y-\delta) < \densityRatio \stepdensity(y-\delta).
\end{align*}
By \eqref{supBUpperBound} and \eqref{supBLowerBound},
\begin{align}
\label{eq:estimating_R_k}
    R_k \le \densityRatio R_k^- + \densityRatio R_k^+,
\end{align}
where
\begin{align}
\label{defOfRkMinus}
    R_k^- &= \int_{A_\delta\cap [D,\infty)} \stepdensity(y-\delta){P}_y^{1,m}\{S_m\in J_k\}dy,
    \\
\label{defOfRkPlus}
    R_k^+  &= \int_{A_\delta\cap (-\infty,D)} \stepdensity(y+\delta){P}_y^{1,m}\{S_m\in J_k\}dy.
\end{align}
One can view these integrals as expectations with respect to the random walk measures $\mathrm{P}_\delta^{0,m}$ and $\mathrm{P}_{-\delta}^{0,m}$. 

Due to \eqref{eq:estimating_R_k} -- \ref{defOfRkMinus}, the estimate \eqref{boundRkGoal} will follow from
\begin{align}
    \sum_{k\in \Z}(R_k^-)^\theta = O(\delta^{-\frac{1}{2}-\varepsilon}m^{-\frac{1}{2}+\varepsilon} + \delta^{\frac{3}{4}-\varepsilon}m^{\frac{7}{20}+\varepsilon}),\label{sumOverRkMinusGoal}\\
    \sum_{k\in \Z}(R_k^+)^\theta = O(\delta^{-\frac{1}{2}-\varepsilon}m^{-\frac{1}{2}+\varepsilon} + \delta^{\frac{3}{4}-\varepsilon}m^{\frac{7}{20}+\varepsilon}).\label{sumOverRkPlusGoal}
\end{align}

Let us prove \eqref{sumOverRkMinusGoal}. 
Note that if $y\in A_\delta\cap [D,\infty)$, then by \eqref{supBUpperBound} and the definition of $A_\delta,$ 
\[\densityRatio \stepdensity(y-\delta) \ge \sup_{|x|<\delta} \stepdensity(y-x) \ge \densityRatio^2 \stepdensity(y).\]
Hence,
$A_\delta\cap [D,\infty)\subset \big\{x\in \R\,:\,\stepdensity(x-\delta) \ge \densityRatio \stepdensity(x)\big\}.$
This and \eqref{defOfRkMinus} imply
\begin{align}
    R_k^- &= \mathrm{P}_\delta^{0,m}\Big\{S_1\in A_\delta\cap [D,\infty),\,\,S_m\in J_k\Big\}\notag \\
    & \le \mathrm{P}_\delta^{0,m}\Big\{p(S_1-\delta) \ge \densityRatio p(S_1),\,\,S_m\in J_k\Big\},\label{Rk_probBound}
\end{align}
where $(S_1,\dots,S_m)$ is the realization of a random walk with distribution  $\mathrm{P}_\delta^{0,m}.$

 The following two lemmas will allow us to give upper bounds on $R_k^-$ and $R_k^+.$ We postpone the proof of both lemmas until the end of this proof. 
 We define $E_t = \{y\in \R\,:\,\stepdensity(y) \ge \densityRatio \stepdensity(y+t)\}$  for  $t\in \R$ and 
 recall that $\lambda(dx) = \stepdensity(x)dx$ denotes the distribution of one step of the random walk.
 \begin{lemma}\label{holderRegOfSpatialShift}
    We have
    \begin{equation}
        \lambda(E_t) = O(|t|),\quad t\to 0.\label{fractional_t_bound}
    \end{equation}
\end{lemma}
\begin{lemma}\label{mStepDensityBound}
    There is a constant $C$ such that for all $z\in\R$, $\delta \in(0,1)$, and $n\in \N$
    \begin{equation}
        \mathrm{P}^n\{S_n\in [z,z+\delta)\} \le \frac{C \delta}{\sqrt{n}(1+z^2n^{-1})}.
    \end{equation}
\end{lemma}

Note that for $t\in\R,$  
\begin{align}
    \mathrm{P}_t^{0,m}\big\{\stepdensity (S_1-t) \ge \densityRatio \stepdensity(S_1))\big\}& = \int_{\{x\in \R\,:\,\stepdensity(x-t)\ge \densityRatio \stepdensity(x)\}} \stepdensity(x-t)dx\notag\\
    & = \int_{\{y\in \R\,:\,\stepdensity(y)\ge \densityRatio \stepdensity(y+t)\}} \stepdensity(y)dy\notag\\
    & = \lambda(E_t). \label{oneStepSpatialShift}
\end{align}
Equality \eqref{oneStepSpatialShift} and \Cref{holderRegOfSpatialShift} imply that there is $C>0$ such that for all $\delta>0$ sufficiently small,
\begin{equation}
    \mathrm{P}_\delta^{0,m}\big\{\stepdensity(S_1-\delta) \ge \densityRatio \stepdensity(S_1)\big\} \le C \delta.\label{oneStepDelta_s}
\end{equation}

We next derive two bounds on the right-hand side of \eqref{Rk_probBound} useful for small and large $k$, respectively. Inequality~\eqref{Rk_probBound}, inequality~\eqref{oneStepDelta_s} and the Markov property of random walks imply
\begin{align}
    R_k^- & = \mathrm{P}_\delta^{0,m}\big\{\stepdensity(S_1-\delta) \ge \densityRatio \stepdensity(S_1)\big\}\cdot  \mathrm{P}_\delta^{0,m}\big\{S_m\in J_k\,|\, \stepdensity(S_1-\delta) \ge \densityRatio \stepdensity(S_1)\big\}\notag\\
     &\le \mathrm{P}_\delta^{0,m}\big\{\stepdensity(S_1-\delta) \ge \densityRatio \stepdensity(S_1)\big\}\cdot \sup_{x\in \R} \mathrm{P}_x^{1,m}\big\{S_m\in J_k\big\}.\notag\\
     & \le C \delta \sup_{x\in \R} \mathrm{P}_x^{1,m}\big\{S_m\in J_k\big\}.\label{supOverxProbBound}
\end{align}
\Cref{mStepDensityBound} implies that for $m\ge 2$
\begin{align*}
    \sup_{x\in \R} \mathrm{P}_x^{1,m}\big\{S_m\in J_k\big\} &= \sup_{x\in \R} \mathrm{P}_x^{m-1}\big\{S_{m-1}\in J_k\big\} \\
    &\le \sup_{x\in \R}\frac{C \delta}{\sqrt{m-1}(1+(k\delta-x)^2(m-1)^{-1})} \\
    &\le C' \delta m^{-1/2}.
\end{align*}
Combining this with \eqref{supOverxProbBound} we obtain
\begin{align}
    R_k^-\le C \delta^{2} m^{-1/2}\label{Rk_supBound}
\end{align}
for some constant $C$. This upper bound is mostly useful for small $k$, since it does not take into account the tail decay of $S_m$. Let us now derive an upper bound useful 
for large $k$. H\"{o}lder's inequality and \eqref{oneStepDelta_s} give us the bound
\begin{align}
    R_k^- &\le \mathrm{P}_\delta^{0,m}\big\{p(S_1-\delta) \ge \densityRatio p(S_1)\big\}^{1/5}\cdot \mathrm{P}_\delta^{0,m}\big\{S_m\in J_k\big\}^{4/5}\notag\\
    & \le C \delta^{1/5} \mathrm{P}_\delta^{0,m}\big\{S_m\in J_k\big\}^{4/5}.\label{CS_Bound}
\end{align}
Let $\alpha > 0$ be a number to be specified later. Applying \eqref{Rk_supBound} and \eqref{CS_Bound} to $|k|\le\delta^{-\alpha}$ and $|k| > \delta^{-\alpha},$ respectively, we obtain
\begin{align}
    \sum_{k\in \Z}(R_k^-)^\theta & = \sum_{|k|\le \delta^{-\alpha}}(R_k^-)^\theta + \sum_{|k|> \delta^{-\alpha}}(R_k^-)^\theta\notag\\
    & \le \sum_{|k|\le \delta^{-\alpha}} C \delta^{2\theta} m^{-\theta/2} + \sum_{|k|> \delta^{-\alpha}}C \delta^{\theta/5} \mathrm{P}_\delta^{0,m}\big\{S_m\in J_k\big\}^{4\theta/5}\notag\\
    & \le C \delta^{2\theta-\alpha} m^{-\theta/2} + C \delta^{\theta/5}\sum_{|k|> \delta^{-\alpha}}\mathrm{P}_\delta^{0,m}\big\{S_m\in J_k\big\}^{4\theta/5}.\label{sumOverRk}
\end{align}
\Cref{mStepDensityBound} implies that for all $k\in \Z\setminus\{0\}$ and $\delta>0,$
\[\mathrm{P}_\delta^{0,m}\big\{S_m\in J_k\big\}^{4\theta/5} \le \frac{C \delta^{4\theta/5} m^{-2\theta/5}}{(1 + (k\delta-\delta)^2m^{-1})^{4\theta/5}} \le \frac{C' \delta^{4\theta/5} m^{-2\theta/5}}{(1 + k^2\delta^2 m^{-1})^{4\theta/5}}\]
for some constant $C'$ independent of $\delta,m,$ and $k$. This implies
\begin{align}
    \sum_{|k|> \delta^{-\alpha}}\mathrm{P}_\delta^{0,m}\big\{S_m\in J_k\big\}^{4\theta/5} & \le C' \delta^{4\theta/5} m^{-2\theta/5} \sum_{|k|> \delta^{-\alpha}} \frac{1}{(1 + k^2\delta^2 m^{-1})^{4\theta/5}}\notag\\
    & \le 2  C' \delta^{4\theta/5} m^{-2\theta/5} \int_{\delta^{-\alpha}-1}^\infty \frac{1}{(1 + y^2\delta^2 m^{-1})^{4\theta/5}}dy\notag\\
    & = 2  C'\delta^{4\theta/5-1} m^{1/2-2\theta/5}  \int_{m^{-1/2}\delta(\delta^{-\alpha}-1)}^\infty \frac{1}{(1 + z^2)^{4\theta/5}}dz. \label{sumOverLargekIntegralBound}
\end{align}
In the above we use the symmetry and monotonicity of $(1+y^2)^{-1}$, and the change of variables $z = y \delta m^{-1/2}$. We can then use Markov's inequality for $r > 0$ and $1/2>\delta>0$ to obtain 
\begin{align}
    \int_{m^{-1/2}\delta(\delta^{-\alpha}-1)}^\infty \frac{1}{(1 + z^2)^{4\theta/5}}dz &\le \frac{1}{(m^{-1/2}\delta (\delta^{-\alpha}-1))^r}\int_{-\infty}^\infty \frac{|z|^r}{(1+z^2)^{4\theta/5}}dz\notag\\
    & \le \frac{C}{m^{-r/2}\delta^{r-r\alpha}}\int_{-\infty}^\infty \frac{|z|^r}{(1+z^2)^{4\theta/5}}dz.\label{MarkovInequalityUpperBound_r}
\end{align}
As long as $r < \frac{8\theta}{5} - 1$, the right-hand side of \eqref{MarkovInequalityUpperBound_r} is finite. Inequalities~\eqref{sumOverLargekIntegralBound} and ~\eqref{MarkovInequalityUpperBound_r} imply that for such $r$, there is a positive constant $C$ such that
\begin{align}
    \sum_{|k|> \delta^{-\alpha}}\mathrm{P}_\delta^{0,m}\big\{S_m\in J_k\big\}^{4\theta/5} \le  C\delta^{4\theta/5-1-r+r\alpha} m^{1/2-2\theta/5+r/2}\label{RkBound_largek}
\end{align}
Displays \eqref{sumOverRk} and \eqref{RkBound_largek} imply that for any $(\theta,r,\alpha)$ satisfying $0<\theta<1,$ $0 < r < \frac{8\theta}{5} - 1$, and $\alpha > 0$, there is a finite constant $C$ such that 
\begin{equation}
    \sum_{k\in \Z}(R_k^-)^\theta  \le C\delta^{2\theta-\alpha} m^{-\theta/2} + C\delta^{\theta-1-r+r\alpha} m^{1/2-2\theta/5+r/2}.\label{sumOverRk_2}
\end{equation}
Set 
\begin{equation}
    \alpha = \frac{5}{2},\qquad r = \frac{1}{2}.\label{alpha_r_def}
\end{equation}
Since $r=\frac{1}{2} < \frac{3}{5}=\frac{8}{5}-1$, we can find $\theta_1\in (0,1)$ such that if $\theta\in (\theta_1,1)$, then 
\[r=\frac{1}{2}  < \frac{8 \theta}{5}-1.\]
Using \eqref{alpha_r_def}  in \eqref{sumOverRk_2}, we obtain 
\begin{equation}
    \sum_{k\in \Z}(R_k^-)^\theta  \le C \delta^{2\theta - \frac{5}{2}}m^{-\theta/2} + C \delta^{\theta - \frac{1}{4}}m^{\frac{3}{4}-\frac{2\theta}{5}}.
\end{equation}
Therefore, for any $\varepsilon>0$,  we can find $\theta_2\in [\theta_1,1)$ such that if $\theta\in (\theta_2,1)$ 
then
~\eqref{sumOverRkMinusGoal} holds. The proof of \eqref{sumOverRkPlusGoal} is similar, and  the lemma follows.
\end{proof}

\begin{proof}[Proof of \Cref{holderRegOfSpatialShift}]
We consider only $t>0$ (the proof for $t<0$ is similar). Suppose that $0<t < R-L$. If $x,x+t \in [L,R]$ then 
$\stepdensity (x) < \densityRatio \stepdensity (x+t)$
by the definition of $\densityRatio$ in \eqref{defOfDensityRatio} and so $x\notin E_t.$ It follows that 
\begin{equation}
    \lambda\big( E_t\cap [L,R]\big) \le \lambda([R-t,R]) \le t \sup_{x\in \R} \stepdensity(x)\label{Et_LR}.
\end{equation}

Now we consider $E_t\cap (-\infty,L].$ If $x\in  (-\infty,L-t]$ and $\stepdensity(x) > 0,$ then $x\notin E_t$ because $\stepdensity$ is nondecreasing on $(-\infty,L].$ Also, if $x\in (L-t,L]$, then $x+t\in [L,R]$ and so $\stepdensity(x+t) \ge \inf_{z\in [L,R]}\stepdensity(z).$ This implies $\stepdensity(x) < \densityRatio \stepdensity(x+t).$ It follows that 
\begin{align}
    \lambda(E_t\cap (-\infty,L]) = 0.\label{Et_InftyL}
\end{align}

Now we consider the set $E_t\cap [R,\infty)$, where $\stepdensity$ is nonincreasing. Define $x_n = R + nt$ for $n\in \N.$ If $\stepdensity(x) \ge \densityRatio \stepdensity(x+t)$ for some $x\in [x_n,x_{n+1}]$, then $x+t\in [x_{n+1},x_{n+2}]$ and so $\stepdensity(x_n)\ge \densityRatio \stepdensity(x_{n+2})$. It follows that 
\begin{align}
    \lambda\left(E_t\cap [R,\infty)\right) & = \int_{[R,\infty)} \1_{\{\stepdensity (x)\ge \densityRatio\stepdensity (x+t)\} }\stepdensity (x)dx\notag\\
    & = \sum_{n=0}^{\infty} \int_{x_n}^{x_{n+1}}\1_{\{\stepdensity (x)\ge \densityRatio\stepdensity (x+t)\} }\stepdensity (x)dx\notag\\
    & \le t \sum_{n=0}^{\infty} \1_{\{\stepdensity (x_n) \ge \densityRatio\stepdensity (x_{n+2})\}}\stepdensity (x_n).\label{EtUpperTail}
\end{align}
Let 
\begin{align*}
C_n = |\{2\le k\le n\,:\,k\text{ even and } \stepdensity (x_{k-2}) \ge \densityRatio \stepdensity (x_{k})\}|=\sum_{\substack{k\in 2\N \\  2\le k\le n}} \1_{\{ \stepdensity (x_{k-2}) \ge \densityRatio \stepdensity (x_{k})\}},\\
B_n = |\{3\le k\le n\,:\,k\text{ odd and } \stepdensity (x_{k-2}) \ge \densityRatio \stepdensity (x_{k})\}|=
\sum_{\substack{k\in 2\N+1 \\  3\le k\le n}} \1_{\{ \stepdensity (x_{k-2}) \ge \densityRatio \stepdensity (x_{k})\}}.
\end{align*}
Due to the monotonicity of $p$,
\begin{align*}
\stepdensity (x_n) &\le 
\densityRatio^{-C_n} \stepdensity (x_0) = \densityRatio^{-C_n} \stepdensity (R), \quad n\in 2\N, \\
\stepdensity (x_n) & \le \densityRatio^{-B_n} \stepdensity (x_1) \le \densityRatio^{-B_n} \stepdensity (R), \quad n\in 2\N+1.
\end{align*}
Thus, using \eqref{EtUpperTail} we obtain
\begin{align}
    \lambda\left(E_t\cap [R,\infty)\right) &\le t \stepdensity (R) \left[\sum_{n \text{ even }} \1_{\{\stepdensity (x_n) \ge \densityRatio \stepdensity (x_{n+2})\}}\densityRatio^{-C_n} + \sum_{n \text{ odd }} \1_{\{\stepdensity (x_n) \ge \densityRatio \stepdensity (x_{n+2})\}}\densityRatio^{-B_n}\right]\notag\\
    & \le 2 t \stepdensity (R) \sum_{m=0}^\infty \densityRatio^{- m} \notag\\
     & \le \frac{2 \densityRatio}{\densityRatio-1}t\stepdensity(R).\label{Et_RInfty}
\end{align}
Finally, \eqref{Et_LR}, \eqref{Et_InftyL}, and \eqref{Et_RInfty} imply that
\begin{align}
    \lambda (E_t) & \le t \sup_{x\in \R} \stepdensity(x) + 0 + \frac{2 \densityRatio}{\densityRatio-1}t \stepdensity(R)\notag 
    = O(t),
\end{align}
and the proof is completed.
\end{proof}

\begin{proof}[Proof of \Cref{mStepDensityBound}]
    Let $\stepdensity_n$ be the density of the centered random variable $S_n = \sum_{i = 1}^n X_i$ where $X_i\sim \stepdensity$ are i.i.d.. We use the non-uniform local limit 
    theorems from~\cite{Shevtsova2017}. Corollary 2 and Corollary 5 of \cite{Shevtsova2017} imply that
    \begin{equation}
        \limsup_{n\to \infty}n^{\min(\momentnumber-2,1)/2}\sup_{x\in \R}(1 + |x|^2) |\sqrt{n}\stepdensity_n(x\sqrt{n}) - g(x\sigma^{-1})| < c < \infty,\label{ShevtsovaResult}
    \end{equation}
    where $g$ is the standard normal and $\sigma^2 = \mathrm{E}[|X_1|^2]$, for some constant $c$ depending on $\momentnumber,$ $\mathrm{E}[|X_1|^2],$ and $\mathrm{E}[|X_1|^\momentnumber]$. Display \eqref{ShevtsovaResult} implies in particular that there is a constant $C$ such that for sufficiently large $n$,     
\begin{equation}
        \stepdensity_n(y)  \le \frac{C}{\sqrt{n}(1 + y^2n^{-1})}\label{density_n_upperBound}
    \end{equation}
    for all $y\in \R.$ Therefore,
    \begin{align*}
        \mathrm{P}^n\{S_n\in [z,z+\delta)\} &= \int_z^{z+\delta} \stepdensity_n(y)dy\\
        & \le \int_z^{z+\delta} \frac{C}{\sqrt{n}(1+y^2 n^{-1})} dy\\
        & \le \delta \sup_{y\in [z,z+\delta)} \frac{C}{\sqrt{n}(1+y^2 n^{-1})}\\
        & \le \frac{2C\delta }{\sqrt{n}(1+z^2n^{-1})}
    \end{align*}
    and the lemma is proved.
\end{proof}

\begin{proof}[Proof of \Cref{boundForW0}]
    Let us fix $\delta \in (0,1)$ and $R > 2$.
    By Jensen's inequality,
\begin{align}
    \sum_{J\in \mathcal{J}^\delta}\left(W_{0,J}^m\right)^\theta & = \sum_{|k|\le R}\left(W_{0,J_k}^m\right)^\theta +  \sum_{|k|>R}\left(W_{0,J_k}^m\right)^\theta\notag\\
    & \le (2 R + 1) \left(\sum_{|k|\le R}W_{0,J_k}^m\right)^\theta + \sum_{|k|>R}\left(W_{0,J_k}^m\right)^\theta\notag\\
    & \le 3 R (W^m)^\theta + \sum_{|k|>R}\left(W_{0,J_k}^m\right)^\theta.\label{W0_sumBound}
\end{align}
Assuming $R>2,$ 
\begin{align}
    \sum_{|k|>R}\E[(W_{0,J_k}^m)^\theta] &\le \sum_{|k|>R}\left(\E[W_{0,J_k}^m]\right)^\theta\notag\\
    & \stackrel{\text{\Cref{thm:expectedNormalizedPartitionFcn}}}{=} \sum_{|k|>R}\left(\mathrm{P}^{m}\big\{S_m\in J_k\big\}\right)^\theta\notag\\
    & \stackrel{\text{\Cref{mStepDensityBound}}}{\le} \sum_{|k|>R}^\infty\frac{C' \delta^{\theta} m^{-\theta/2}}{(1 + k^2\delta^2 m^{-1})^{\theta}}\notag\\
    & \le 2C'\delta^\theta m^{-\theta/2} \int_{R-1}^{\infty} \frac{1}{(1 + y^2\delta^2 m^{-1})^\theta}dy\notag\\
    & = 2C'\delta^{\theta-1} m^{1/2-\theta/2}\int_{m^{-1/2}\delta(R-1)}^{\infty} \frac{1}{(1 + z^2)^\theta}dz.\label{W0_IntegralUpperBound}
\end{align}
In the last line of \eqref{W0_IntegralUpperBound}, we use the substitution $z = y\delta m^{-1/2}.$
Next, letting $r>0$ to be chosen later, we can use Markov's inequality to get the bound 
\begin{align}
    \int_{m^{-1/2}\delta(R-1)}^{\infty} \frac{1}{(1 + z^2)^\theta}dz &\le \frac{1}{m^{-r/2}\delta^r(R-1)^r} \int_{-\infty}^\infty \frac{|z|^r}{(1 + z^2)^\theta}dz\notag \\
    & \le C m^{r/2}\delta^{-r} R^{-r}.\label{W0_MarkovUpperBound}
\end{align}
In order for the integral $\int_{-\infty}^\infty \frac{|z|^r}{(1 + z^2)^\theta}dz$ in \eqref{W0_MarkovUpperBound} to be finite we need $r < 2 \theta - 1$. Combining \eqref{W0_sumBound}, \eqref{W0_IntegralUpperBound}, and \eqref{W0_MarkovUpperBound} we obtain 
\begin{equation}
    \sum_{J_k\in \mathcal{J}^\delta}\E[(W_{0,J_k}^m)^\theta] \le 3 R \E[(W^m)^\theta] + C \delta^{\theta-1-r} m^{\frac{1}{2}+\frac{r}{2}-\frac{\theta}{2}}R^{-r}.\label{W0_RflexibleBound}
\end{equation}
Set 
\[R = \left(\frac{\delta^{\theta-1-r} m^{\frac{1}{2}+\frac{r}{2}-\frac{\theta}{2}}}{\E[(W^m)^\theta]}\right)^{\frac{1}{1+r}}.\]
 We have $\E[(W^m)^\theta] \le 1$ by Jensen's inequality, and $\delta < 1$ by assumption, and so $R>2$ for large enough $m$. This choice of $R$ and \eqref{W0_RflexibleBound} imply
\begin{align}
    \sum_{J_k\in \mathcal{J}^\delta}\E[(W_{0,J_k}^m)^\theta] \le C \left(\delta^{\theta-1-r} m^{\frac{1}{2}+\frac{r}{2}-\frac{\theta}{2}}\left(\E[(W^m)^\theta]\right)^{r}\right)^{\frac{1}{1+r}}.\label{theta_r_upperBound}
\end{align}
Let $\varepsilon \in (0,1)$. Set $r = 1- \varepsilon$ and $\theta_0 = 1 - \frac{1}{2}\varepsilon$. Now if $\theta\in(\theta_0,1)$, then 
\[1-\varepsilon = r < 2\theta-1\] 
and so the integral in \eqref{W0_MarkovUpperBound} is finite. Finally, with this choice of $r$ and $\theta_0$, \eqref{theta_r_upperBound} implies that for all $\theta\in(\theta_0,1)$,
\eqref{eq:boundForW0} holds, which completes the proof.
\end{proof}

\subsection{Decay of Fractional Moments}\label{decayOfFractionalMoments}

The main purpose of this section is to prove \Cref{lem:fractionalMomentDecay}. Our study  
of the overlap metric
on probability distributions in
Section~\ref{overlapMetricAnalysis}
is of independent interest.

Our proof of \Cref{lem:fractionalMomentDecay} is based on the following lemma, proved in Section~\ref{recursiveBoundSection}.
\begin{lemma}\label{recursiveBound}
    For any $\theta\in(0,1)$, there are positive constants $C_1,C_2$ such that 
    \[\E[(W^{n+1})^\theta] \le \left(1 - \frac{C_1}{t}\right)\E[(W^{n})^\theta] + \frac{C_2}{t}\left(\mathrm{P}^n\big\{|S_n|\ge t\big\}\right)^\theta,
     \quad t\ge 1, \  n\in \N.\]
\end{lemma}

\begin{proof}[Proof of \Cref{lem:fractionalMomentDecay}]
 Let $\alpha = \momentnumber /2 - \epsilon.$ If $\alpha \le 0$ then the lemma is easily verified because 
    \[\E[(W^m)^\theta] \le (\E [W^m])^\theta = 1,\quad m\in\N.\]
  We now consider $\epsilon$ small enough so that $\alpha > 0.$

    Using Markov's inequality and the Marcinkiewicz--Zygmund inequality (see \cite{Marcinkiewicz1937} or Chapter 5 of \cite{PetrovValentinV1975SoIR}),
    we obtain
    \begin{align*}
        \mathrm{P}^n\big\{|S_n|\ge t\big\} &\le \mathrm{E}^n[S_n^\momentnumber] {t^{-\momentnumber}}
\le C_\momentnumber n^{\momentnumber/2}t^{-\momentnumber}.
    \end{align*}
    Thus \Cref{recursiveBound}
    implies that there are constants $C_1,C$ such that
    \begin{equation}\label{momentFractionalUpperBound}
        \E[(W^{n+1})^\theta] \le \left(1 - C_1 t^{-1}\right)\E[(W^{n})^\theta] + C t^{-1-\momentnumber\theta} n^{\momentnumber\theta /2}.
    \end{equation}
Let us fix $\theta_0\in(0,1)$ and $\beta \in(0,1)$ such that
    \begin{equation}
    \alpha_2 := \beta + \momentnumber \theta \beta- \frac{1}{2}\momentnumber \theta > 1 + \alpha\label{betaThetaAlphaRelation}
    \end{equation}
    holds for all $\theta\in(\theta_0,1)$. Consider the sequence given by $t = t_n = n^{\beta}.$ We have by \eqref{momentFractionalUpperBound}, for $\theta\in (\theta_0,1),$
    \begin{align}
        \E[(W^{n+1})^\theta] &\le (1 - C_1 n^{-\beta})\E[(W^{n})^\theta] + C n^{-\alpha_2}.\label{recursiveBound1}
    \end{align}
    We now prove that there is $N\in \N$ such that for all $K>1$ and $n>N$,
    \begin{equation}
        K(1 - C_1 n^{-\beta})n^{-\alpha} + C n^{-\alpha_2} \le K(n+1)^{-\alpha}.\label{elementaryRecursiveInequality}
    \end{equation}
    Since $\beta < 1$ and $\alpha - \alpha_2 < -1$, there is $N\in \N$ such that the inequality 
    \begin{equation}
        C n^{\alpha-\alpha_2 } \le K C_1 n^{-\beta} - \alpha K n^{-1}\label{nLargeInequality}
    \end{equation}
    holds for all $n > N$ and all $K>1$. Using the Taylor expansion identity 
    \[(1 + n^{-1})^{-\alpha} = 1 - \alpha n^{-1} + o(n^{-1})\] 
    and rearranging \eqref{nLargeInequality}, we conclude that we can adjust $N$ so that
    \begin{equation*}
        C n^{\alpha - \alpha_2} \le -K (1 - C_1 n^{-\beta}) + K(1 + n^{-1})^{-\alpha}
    \end{equation*}
    for all $n>N$ and all $K>1$. Multiplying both sides of this inequality
    by $n^{-\alpha}$ and rearranging we obtain that \eqref{elementaryRecursiveInequality} holds for all $n\ge N$. 
    
    Let us fix  this $N$ and define
 $K = \max(1,N^{\alpha}\E[(W^N)^\theta])$. Note that $\E[(W^{N})^\theta] \le K N^{-\alpha}.$ Then, \eqref{recursiveBound1} and \eqref{elementaryRecursiveInequality} imply
    \begin{align*}
        \E[(W^{N+1})^\theta] &\le (1 - C_1 N^{-\beta})\E[(W^{N})^\theta]+ C N^{-\alpha_2} \\
        & \le K(1 - C_1 N^{-\beta})N^{-\alpha} + C N^{-\alpha_2} \\
        & \le K(N+1)^{-\alpha}.
    \end{align*}
    Using this computation as an induction step, one can then show that 
    \[\E[(W^n)^\theta] \le K n^{-\alpha}\]
    for all $n > N,$ which proves the lemma because $\alpha = \momentnumber/2-\epsilon.$
\end{proof}

\subsubsection{Analysis of $I(r,\mu)$}\label{overlapMetricAnalysis}
First we define the following measurement of concentration of a probability measure $\mu.$
\begin{definition}
    For a probability measure $\mu$ on $\R^d$, let 
    \[I(r,\mu) = \mu^{\otimes 2}\{(X,Y)\in(\R^d)^2:|X-Y|<r\},\qquad r>0.\]

\end{definition}
Here and below, 
$(X,Y)$ is the canonical pair of i.i.d.\ r.v.'s with distribution $\mu$.

 \Cref{thm:probComparison} gives a bound on how $I(\mu,r)$ changes with~$r$ and compares it to the concentration function of $\mu$. 
This comparison is a continuous version of display (2.8) in \cite{CSY03} valid for the lattice setting. We will need notation for Euclidean balls:
\[B(x,r)=\{y\in\R^d:\ |y-x|<r\},\quad x\in\R^d,\ r>0.\]
\begin{lemma}\label{thm:probComparison}
Let $d\in\N$.    There are constants $c=c_d,$ $C=C_d,$ and $C_d'=C'$ such that for all $R,r\in \R$ satisfying $R\ge r>0$ and for all probability measure $\mu$ on $\R^d$,
    \begin{equation}
        I(r,\mu) \le I(R,\mu) \le C \frac{R^d}{r^d}I(r, \mu)\label{eq:probComparison}
    \end{equation}
    and 
    \begin{equation}
        c\sup_{x\in \R^d}\mu\left(B(x,r)\right)^2 \le I(r,\mu) \le C' \sup_{x\in \R^d}\mu\left(B(x,r)\right). \label{eq:comparisonWithMax}
    \end{equation}
\end{lemma}
\begin{proof}
Let us prove \eqref{eq:probComparison} first.
    The first inequality follows from the inclusion 
    \[\{(x,y)\in (\R^d)^2\,:\,|x-y|<r\}\subset \{(x,y)\in (\R^d)^2\,:\,|x-y|<R\},\quad R\ge r.\]
To prove the second inequality in~\eqref{eq:probComparison}, it suffices to prove
    \begin{equation}
        \mu^{ \otimes 2}\{|X-Y|< Kr\} \le C K^d \mu^{ \otimes 2}\{|X-Y|<r\}\label{simplifiedLdeltaInequality}
    \end{equation}
    for $K = \lceil \frac{R}{r}\rceil \in \N$ and some constant $C$ independent of $\mu$ and $K$, 
    because $K \le 2\frac{R}{r}.$ 

    Let $\mathcal{J}^r$ denote collection of disjoint cubes $J_{\vec{k}}^r$ of diagonal length $r$ of the form 
    \[\vec{k}r d^{-1/2} + [0,rd^{-1/2})^d,\quad \vec{k}\in \Z^d.\]
Note that if $X,Y\in J_{\vec{k}}^r$ then $|X-Y| < r.$ It follows that for all $x_0\in \R^d,$
    \begin{align}
        \mu^{\otimes 2}\{|X-Y|<r\} & \ge \mu^{\otimes 2}\bigg( \bigcup_{\vec{k}\in \Z^d}\{X,Y\in J_{\vec{k}}^r+x_0\}\bigg)\notag\\
        & = \sum_{\vec{k}\in \Z^d} \mu\Big(J_{\vec{k}}^r + x_0 \Big)^2
        .\label{unionOfIntervalsLowerBound}
    \end{align}

    Inequality \eqref{unionOfIntervalsLowerBound} with $x_0 = 0$ and the Cauchy-Schwarz inequality implies
    \begin{align}
        \mu^{\otimes 2}\{|X-Y|<r\} &\ge \sum_{\vec{k}\in \Z^d} \sum_{\vec{\ell}\in [0,K)^d\cap \Z^d} \mu\left(J_{K \vec{k}+\vec{\ell}}^r\right)^2\notag\\
        & \ge \sum_{\vec{k}\in \Z^d} \frac{1}{K^d}\bigg(\sum_{\vec{\ell}\in [0,K)^d\cap \Z^d} \mu\big(J_{K \vec{k}+\vec{\ell}}^r\big)\bigg)^2\notag\\
        & = \frac{1}{K^d}\sum_{\vec{k}\in \Z^d} \mu\left(J_{\vec{k}}^{r K}\right)^2.\label{overlapLowerBound}
    \end{align}
    Also, if $|X-Y|<r K$, then there is some $\vec{k}\in\Z^d$ such that $X\in J_{\vec{k}}^{r K}$ and $Y\in J_{\vec{k}'}^{r K}$ for some $\vec{k}'$ within distance $\sqrt{d}$ of $\vec{k}$. For $\vec{k}\in \Z^d,$ let $N(k) = \{\vec{k}'\in \Z^d\,:\,|\vec{k}-\vec{k}'|\le \sqrt{d}\}.$ It follows that 
    \begin{align}
        \mu^{\otimes 2}\{|X-Y|<Kr\} &\le \mu^{\otimes 2}\Big( \bigcup_{\vec{k}\in \Z^d}\Big\{X,Y\in \bigcup_{\vec{k}'\in N(k)} J_{\vec{k}'}^{r K}\Big\}\Big)\notag\\
        \notag 
        & \le \sum_{\vec{k}\in \Z^d} \mu\Big( \bigcup_{\vec{k}'\in N(k)}J_{\vec{k}'}^{r K}\Big)^2\\
         & \le \sum_{\vec{k}\in \Z^d} \bigg(\sum_{\vec{k'}\in N(k)}\mu\big(J_{\vec{k}'}^{r K}\big)\bigg)^2\notag\\
        & \le |N(0)| \sum_{\vec{k}\in \Z^d} \sum_{\vec{k}'\in N(k)}\mu\big(J_{\vec{k}'}^{r K}\big)^2\notag\\
        & = |N(0)|^2 \sum_{\vec{k}\in \Z^d} \mu\big(J_{\vec{k}}^{r K}\big)^2.\label{overlapUpperBound}
    \end{align}
    Inequalities \eqref{overlapLowerBound} and \eqref{overlapUpperBound} imply
    \[\mu^{\otimes 2}\{|X-Y|<Kr\} \le |N(0)|^2 K^d \mu^{\otimes 2}\{|X-Y|<r\},\]
    which proves \eqref{simplifiedLdeltaInequality} and establishes \eqref{eq:probComparison}.

    Now we establish \eqref{eq:comparisonWithMax}. Let $x\in \R^d$ and let $y = x - (r,\dots, r).$ Inequality \eqref{unionOfIntervalsLowerBound} applied to  $\sqrt{d} r$ 
    instead of $r$ and $x_0=y$, and the inclusion 
    \[B(x,r)\subset J_{\vec{0}}^{2\sqrt{d}r} + y\]
    imply
    \begin{align*}
        \mu^{\otimes 2}\{|X-Y|<2\sqrt{d}r\} \ge \mu\left(J_{\vec{0}}^{2\sqrt{d}r} + y\right)^2  \ge \mu(B(x ,r))^2.
    \end{align*}
    This and \eqref{eq:probComparison} imply 
    \[\mu(B(x,r))^2 \le \mu^{\otimes 2}\{|X-Y|<2\sqrt{d}r\}  \le C_d 2^dd^{d/2} \mu^{\otimes 2}\{|X-Y|<r\},\]
    which establishes the lower bound in \eqref{eq:comparisonWithMax}. For the upper bound, we apply
 \eqref{overlapUpperBound} with $K=1$:
    \begin{align}
        \mu^{\otimes 2}\{|X-Y|< r\} &\le |N(0)|^2 \sup_{\vec{k}\in \Z^d} \mu\big(J_{\vec{k}}^{r}\big) \sum_{\vec{k}\in \Z^d} \mu\big(J_{\vec{k}}^{r}\big)\notag \\
        & =  |N(0)|^2 \sup_{\vec{k}\in \Z^d} \mu\big(J_{\vec{k}}^{r}\big)\notag \\
        & \le |N(0)|^2   \sup_{x\in \R^d}\mu\left(B(x,r)\right),\label{overlapMaxUpperBound}
    \end{align}
and the proof is complete.
\end{proof}

Let $\mathscr{C}$ be the space of continuous functions from $\R^d$ to $\R$ equipped with the topology of uniform convergence on compact sets and let $\mathbb{Q}$ be a measure on the Borel $\sigma$-algebra of $\mathscr{C}$. The symbol $\eta(x)$ will represent the random variable $\eta\mapsto \eta(x)$ defined on $\mathscr{C}.$ We next prove a lemma relating $I(r,\mu)$ to the random variable $\eta\mapsto \int_{\R^d}\eta(x)\mu(dx)$ defined on $\mathscr{C}.$ This is a continuous setting extension of the lattice setting Lemma 3.1 of \cite{CSY03}.
\begin{lemma}\label{thm:U2bound}
    Suppose $\mathbb{Q}$ satisfies the following conditions:
    \begin{requirements}
        \item The stochastic process $(\eta(x))_{x\in \R^d}$ is stationary with respect to spatial shifts.\label{stationaryEtaReq}
        \item $\mathbb{Q}[|\eta(0)|^3]< \infty$ and $\mathbb{Q}[\eta(0)] = 0$.\label{thirdMomentEtaReq}
        \item If $R(x) = \mathbb{Q}[\eta(0)\eta(x)]$ and $W(x,y) = \E[\eta(0)\eta(x)\eta(y)]$ for $x,y\in \R^d$, then $R$ is nonnegative, is bounded away from zero in a neighborhood around the origin, and is compactly supported. $W$ is compactly supported.\label{RWReq}
    \end{requirements}
Then for every $r > 0$ there are constants $c_1,c_2,c_3>0$ such that the following holds for all probability measures $\mu$ on $\R^d$. If $U = \int_{\R^d} \eta(x)\mu(dx)$, then 
    \begin{equation}\label{eq:L2comparison}
        c_1  I(r,\mu)\le \mathbb{Q}\left[ U^2\right] \le c_2 I(r,\mu)
    \end{equation}
    and
    \begin{equation}\label{eq:U2LowerBound}
        \mathbb{Q}\left[ \frac{U^2}{2 + U}\right] \ge c_3 I(r,\mu).
    \end{equation}
\end{lemma}
\begin{proof}
Due to \Cref{thm:probComparison}, it suffices to prove   
\eqref{eq:L2comparison},\eqref{eq:U2LowerBound} for one 
value $r=r_0$ of our choice. 
Using Requirements~\ref{thirdMomentEtaReq} and \ref{RWReq}, we can find a number $r_0>0$ and 
numbers $C_1,C_2,C_3,K,K' > 0$ such that for all $x,y\in \R^d$
    \begin{align}
        C_1\1_{|x|<r_0} &\le R(x),\label{RlowerBound}\\
        C_2\1_{|x|< K} &\ge R(x),\label{RupperBound}\\
        C_3 \1_{|x|<K'} &\ge W(x,y).\label{WupperBound}
    \end{align}
     Fubini's theorem and stationarity of $\eta$ imply
    \begin{align}
        \mathbb{Q}\left[U^2\right]  =  \mathbb{Q}\left[\Big(\int_{\R^d} \eta(x)\mu(dx)\Big)^2\right]
        = \int_{(\R^d)^2} R(x-y) \mu(dx)\mu(dy).\label{U2_Rformulation}
    \end{align}
    Inequality \eqref{RlowerBound} and equality \eqref{U2_Rformulation} imply
    \begin{equation}
        \mathbb{Q}\left[U^2\right] \ge C_1 \mu^{\otimes 2}\{|X-Y|<r_0\} = C_1 I(r_0,\mu).\label{U2lowerBound2}
    \end{equation}
    Inequality \eqref{RupperBound} and equality \eqref{U2_Rformulation} imply
    \begin{equation}
        \mathbb{Q}\left[U^2\right] \le C_2 \mu^{\otimes 2}\{|X-Y|<K\} = C_2 I(K,\mu).\label{U2upperBound2}
    \end{equation}
    \Cref{thm:probComparison} then implies 
    \[I(K,\mu) \le C I(r_0,\mu)\]
    for some constant $C$ independent of $\mu.$ \Cref{eq:L2comparison} is then proved for this $r_0$, and thus for any $r$ with possibly different constants. 

Similarly,
    \begin{align}
        \mathbb{Q}\left[U^3\right] &=  \int_{(\R^d)^3} W(x-z,y-z) \mu(dx)\mu(dy)\mu(dz)\notag\\
        & \stackrel{\eqref{WupperBound}}{\le}  C_3 I(K',\mu)\notag\\
        & \stackrel{\text{\Cref{thm:probComparison}}}{\le} C' I(r_0,\mu).\label{U3_Wformulation}
    \end{align}
    Now we can use exactly the same argument as in \cite{CSY03}:
    \begin{align*}
        C_1 I(r_0,\mu) &\stackrel{\eqref{U2lowerBound2}}{\le} \mathbb{Q}\left[\frac{U}{\sqrt{2+U}}U\sqrt{2+U}\right]\\
        & \le \left(\mathbb{Q}\left[\frac{U^2}{2+U}\right]\right)^{1/2}\left(\mathbb{Q}\left[2U^2 + U^3\right]\right)^{1/2}\\
        & \stackrel{\eqref{U2upperBound2},\eqref{U3_Wformulation}}{\le} C'' I(r_0,\mu)^{1/2}\left(\mathbb{Q}\left[\frac{U^2}{2+U}\right]\right)^{1/2}.
    \end{align*}
    \Cref{eq:U2LowerBound} then follows for any $r > 0$ with a possibly different constant.
\end{proof}

\subsubsection{Proof of \Cref{recursiveBound}}\label{recursiveBoundSection}

Recall that $\rho^n$ is the polymer endpoint measure at step~$n$. The lattice version of \Cref{thm:expofWILowerBound} can be found in Lemma 4.2 of \cite{CSY03}.
\begin{lemma}\label{thm:expofWILowerBound}
    For any $\theta\in (0,1)$ and $R,r>0,$ satisfying $R\ge r$ there are positive constants $C_1,C_2$ such that 
    \begin{equation}
        \E[(W^{n})^\theta I(r,\rho^{n})] \ge \frac{C_1 r}{R}\E[(W^n)^\theta] - \frac{C_2 r }{R}\left(\mathrm{P}^n\big\{|S_n|\ge R\big\}\right)^\theta.\label{WILowerBoundEq}
    \end{equation}
\end{lemma}
\begin{proof}
    Inequalities \eqref{eq:probComparison} and \eqref{eq:comparisonWithMax} of \Cref{thm:probComparison} imply that for some constants $C$ and $C'$ depending on neither $r,R$ nor $\rho^n$,
    \begin{equation}
        I(r, \rho^n) \ge C  \frac{r}{R}I(2 R,\rho^n) \ge C'\frac{r}{R}\rho^n([-R,R))^2. \label{ILowerBoundRho}
    \end{equation}
        Since $\rho^n([-R,R)^c) \in [0,1]$,
        \begin{align}
            \rho^{n}([-R,R))^2 & = (1 - \rho^{n}([-R,R)^c))^2\notag\\
            & \ge 1-2\rho^{n}([-R,R)^c)\notag\\
            & \ge 1-2\rho^{n}([-R,R)^c)^\theta.\label{RhoLowerBound}
        \end{align}
        Using Jensen's inequality and \Cref{thm:expectedNormalizedPartitionFcn} we obtain
        \begin{align}
            \E[(W^{n})^\theta \rho^{n}([-R,R))^\theta] & \le (\E[W^{n}\rho^n([-R,R))])^\theta\notag\\
            & = (\E [W_{0,[-R,R)^c}^{0,n}] )^\theta\notag\\
            & = \left(\mathrm{P}^{n}\big\{|S_n|>R\big\}\right)^\theta.\label{WnRhoUpperBound}
        \end{align}
        Putting \eqref{ILowerBoundRho}, \eqref{RhoLowerBound}, and \eqref{WnRhoUpperBound} together, we obtain \eqref{WILowerBoundEq}.
\end{proof}

\begin{proof}[Proof of \Cref{recursiveBound}]
We are going to apply \Cref{thm:U2bound}, so we set
\[\eta(x) = e^{- F_n(x)-\logExp}-1\]
and define the measure $\QQ$ on $\mathscr{C}$ to be the marginal of $\Prb$ associated to the $n$-th time coordinate. Equivalently, $\QQ$ 
is a version of the conditional expectation given $(F_k)_{k\neq n}$.
  \Cref{continuousEnvironmentAssumptions} implies that $\Prb$-a.s., $\QQ$ satisfies the conditions of \Cref{thm:U2bound}.

Since    
    \begin{align*}
        \frac{W^{n+1}}{W^n}& = \frac{1}{W^n}\int_\R \int_{\R} W_x^n e^{-F_{n}(x)-\logExp}\stepdensity(y-x)dxdy\notag\\
        & = \frac{1}{W^n}\int_\R W_x^n e^{-F_{n}(x)-\logExp} dx\notag\\
        & = \int_{\R}e^{- F_n(x)-\logExp} \rho^n(dx),
    \end{align*}       
we obtain \[U = \int_\R \eta(x)\mu(dx)\]
for $\mu = \rho^n$ and $U = \frac{W^{n+1}}{W^n}-1$, and we can apply
\Cref{thm:U2bound}.

Equation \eqref{eq:U2LowerBound} of \Cref{thm:U2bound} implies that there is a deterministic constant $C$ such that $\Prb$-a.s.,
    \begin{equation}
        \QQ\left[\frac{U^2}{2+U}\right] \ge C I(1,\rho^n).\label{conditionalU2LowerBound}
    \end{equation}
Since $\theta\in(0,1)$, there is $c>0$ such that
    \begin{equation}
        (u+1)^\theta-\theta u-1 \le -c \frac{u^2}{2+u},\quad u \ge -1,
        \label{uthetaDeterministicBound}
    \end{equation}
(see equation (4.5) in \cite{CSY03}). Since $\QQ[U] = 0$,  for all $t\ge 1,$ we have
    \begin{align*}
        \E[(W^{n+1})^\theta - (W^n)^\theta] & = \E\left[ (W^{n})^\theta \QQ[(U+1)^\theta - 1] \right]\\
        & = \E\left[ (W^{n})^\theta \QQ[(U+1)^\theta - \theta U- 1] \right]\\
        & \stackrel{\eqref{uthetaDeterministicBound}}{\le} -c\E\left[ (W^{n})^\theta \QQ\left[\frac{U^2}{2+U} \right]\right]\\
        & \stackrel{\eqref{conditionalU2LowerBound}}{\le} - C \E\left[ (W^{n})^\theta I(1, \rho^n)\right]\\
        & \stackrel{\text{\Cref{thm:expofWILowerBound}}}{\le} -\frac{C_1}{t}\E[(W^n)^\theta] + \frac{C_2}{t}\left(\mathrm{P}^n\big\{|S_n|\ge t\big\}\right)^\theta.
    \end{align*}
    Then,
    \begin{align*}
        \E[(W^{n+1})^\theta] &= \E[(W^n)^\theta] + \E[(W^{n+1})^\theta - (W^n)^\theta]\\
        & \le \left(1 - \frac{C_1 }{t}\right)\E[(W^{n})^\theta] + \frac{C_2}{t}\left(\mathrm{P}^n\big\{|S_n|\ge t\big\}\right)^\theta
    \end{align*}
    as claimed.
\end{proof}

\bibliographystyle{alpha} 
\bibliography{Burgers,polymer}

\end{document}